\documentclass[final,a4]{siamltex}
\usepackage{epsfig}
\usepackage{amsfonts}
\usepackage{subfigure}
\usepackage{amsmath}
\usepackage{color}
\usepackage{fancyheadings}
\usepackage{url}
\usepackage[a4paper,ps2pdf]{hyperref}

\def \R {\mathbb{R}}
\def \C {\mathbb{C}}
\def \Ratn {\mathrm{Rat}^n}
\def \RatN {\mathrm{Rat}^N}

\def \hc {{\widehat c}}
\def \hb {{\widehat b}}
\def \ha {{\widehat \lambda}}
\def \hC {{\widehat C}}
\def \hB {{\widehat B}}
\def \hA {{\widehat A}}
\def \hP {{\widehat P}}
\def \hx {{\widehat x}}
\def \hQ {{\widehat Q}}
\def \hy {{\widehat y}}
\def \hH {{\widehat H}}

\def \calJ {{\cal{J}}}
\def \calH {{\cal{H}}}

\def \calU {{\cal{U}}}
\def \calV {{\cal{V}}}

\newcommand{\e}{\begin{equation}}
\newcommand{\ee}{\end{equation}}
\newcommand{\eqn}{\begin{eqnarray}}
\newcommand{\eeqn}{\end{eqnarray}}
\newcommand{\eqnn}{\begin{eqnarray*}}
\newcommand{\eeqnn}{\end{eqnarray*}}
\newcommand{\bma}{\left[\begin{array}}
\newcommand{\ema}{\end{array}\right]}

\def \diag {\mathrm{diag}}
\def \trace {\mathrm{tr}}

\newtheorem{thrm}{Theorem}[section]
\newtheorem{lmm}[thrm]{Lemma}
\newtheorem{cor}[thrm]{Corollary}
\newtheorem{dfntn}[thrm]{Definition}
\newtheorem{rmrk}[thrm]{Remark}

\usepackage{scrtime}
\usepackage[time]{prelim2e}
\renewcommand{\PrelimWords}{%
File \jobname
}

\newcommand{\pacomm}[1]{}

\title{$\calH_2$-optimal approximation of MIMO linear dynamical systems\thanks{This 
paper presents research supported by the Belgian Network DYSCO (Dynamical Systems, 
Control, and Optimization), funded by the Interuniversity Attraction Poles Programme, 
initiated by the Belgian State, Science Policy Office and by the National Science Foundation
under contract OCI-03-24944. The scientific responsibility rests with its authors.}}

\author{Paul Van Dooren\footnotemark[2]\ \footnotemark[3]\and Kyle A. Gallivan\footnotemark[5]\ \footnotemark[6]
\and P.-A. Absil\footnotemark[2]\ \footnotemark[4]}

\begin{document}
\maketitle

\renewcommand{\thefootnote}{\fnsymbol{footnote}}
\footnotetext[2]{Universit\'e catholique de Louvain (UCL),
Department of Mathematical Engineering, B\^atiment Euler, Avenue
Georges Lema\^itre 4,
B-1348 Louvain-la-Neuve, Belgium. }
\footnotetext[3]{\url{http://www.inma.ucl.ac.be/~vdooren/}} 
\footnotetext[5]{Department of Mathematics, Florida State University, Tallahassee FL 32306, USA.}
\footnotetext[6]{\url{http://www.scs.fsu.edu/~gallivan/}} 
\footnotetext[4]{\url{http://www.inma.ucl.ac.be/~absil/}} 
\renewcommand{\thefootnote}{\arabic{footnote}}

\bibliographystyle{siam}


\centerline{\tt 
Submitted on 30 JUL 2008
}

\begin{abstract}
  We consider the problem of approximating a multiple-input
    multiple-output (MIMO) $p\times m$ rational transfer function
  $H(s)$ of high degree by another $p\times m$ rational transfer
  function $\hH(s)$ of much smaller degree, so that the
  $\calH_2$ norm of the approximation error is minimized. We
  characterize the stationary points of the $\calH_2$ norm of
    the approximation error by tangential interpolation conditions and
  also extend these results to the discrete-time case.  We analyze
  whether it is reasonable to assume that lower-order models can
  always be approximated arbitrarily closely by imposing only
  first-order interpolation conditions. Finally, we analyze the
  $\calH_2$ norm of the approximation error for a simple case in order
  to illustrate the complexity of the minimization problem.
\end{abstract}



\pagestyle{myheadings} \thispagestyle{plain} \markboth{P. VAN DOOREN
  AND K. GALLIVAN AND P.-A. ABSIL}{$H_2$-OPTIMAL APPROXIMATION OF MIMO
LINEAR DYNAMICAL SYSTEMS}

\date{Compact report, version \today}

\section{Introduction} \label{parameter}

In this paper, we consider the problem of approximating a real $p\times m$ rational 
transfer function $H(s)$ of McMillan degree $N$ by a real $p\times m$ rational transfer 
function $\hH(s)$ of lower McMillan degree $n$ using the $\calH_2$-norm 
as the approximation criterion. We refer, e.g.,
to~\cite{Che1999LST,Ant2005} for the relevant background on linear system
theory and model reduction.

Since a transfer function has an unbounded $\calH_2$-norm if 
it is not strictly proper, we will constrain both $H(s)$ and 
$\hH(s)$ to be strictly proper (i.e., they are zero at
$s=\infty$). Such transfer functions have minimal (i.e., controllable
and observable) state-space realizations  
$(A,B,C)\in \R^{N\times N}\times\R^{N\times m}\times\R^{p\times N}$ and 
$(\hA,\hB,\hC)\in \R^{n\times n}\times\R^{n\times m}\times\R^{p\times n}$ satisfying
\begin{equation}\label{eq:LTI}
\left\{\begin{array}{l}
\dot x = Ax + Bu, \\
y = Cx,
\end{array}\right. \quad H(s):=C(sI_N-A)^{-1}B, 
\end{equation}
and
\begin{equation}\label{eq:LTIhat}
\left\{\begin{array}{l}
\dot{\hx} = \hA\hx + \hB u, \\
\hy = \hC\hx,
\end{array}\right. \quad \hH(s):=\hC(sI_n-\hA)^{-1}\hB ,
\end{equation}
where $u\in \R^m$, $y,\hy\in \R^p$, $x\in \R^N$, $\hx\in \R^n$.

We also look at the equivalent formulation in the discrete-time case where 
the dynamical systems become
\begin{equation}\label{eq:LTID}
\left\{\begin{array}{l}
x_{k+1} = Ax_k + Bu_k\\y_k= Cx_k
\end{array}\right. \quad H(z):=C(zI_N-A)^{-1}B, 
\end{equation}
and
\begin{equation}\label{eq:LTIDhat}
\left\{\begin{array}{l}
\hx_{k+1} = \hA \hx_k +\hB u\\\hy_k= \hC \hx_k
\end{array}\right. \quad \hH(z):=\hC(zI_n-\hA)^{-1}\hB.
\end{equation}

Expressions for the gradients of the squared $\calH_2$-norm error
  function 
\[
\calJ_{(A,B,C)}: (\hA,\hB,\hC)\mapsto \|C(sI_N-A)^{-1}B -
  \hC(sI_n-\hA)^{-1}\hB\|_{\calH_2}^2
\]
have been known since the work of Wilson~\cite{Wil1970} (the
expressions are recalled in Theorem~\ref{gradientsJ}). One can
object, however, that the \emph{full parameterization}
\begin{equation}  \label{eq:full-param}
(\hA,\hB,\hC)\mapsto
\hH(s)=\hC(sI_n-\hA)^{-1}\hB
\end{equation}
is not one to one, since the triple
\[
(\hA_T,\hB_T,\hC_T):=(T^{-1}\hA T,T^{-1}\hB,\hC T)
\] 
for any matrix $T \in GL(n,\R)$ defines the same transfer function~:
\[\hH(s)=\hC(sI_n-\hA)^{-1}\hB=
\hC_T(sI_n-\hA_T)^{-1}\hB_T,\]or
\[\hH(z)=\hC(zI_n-\hA)^{-1}\hB=
\hC_T(zI_n-\hA_T)^{-1}\hB_T.\] If one could eliminate the $n^2$
degrees of freedom of the invertible transformation $T$, one could
hope to fully parameterize the target system $\hH(s)$ or $\hH(z)$ with
only $n(m+p)$ independent parameters, and to turn Wilson's conditions
into $n(m+p)$ nonredundant scalar conditions.  Concerning the
parameterization task, Byrnes and Falb~\cite[Th.~4.7]{ByrFal1979} show
that the set $\Ratn_{p,m}$ of $p\times m$ strictly proper rational
transfer functions of degree $n$ can be parameterized with only
$n(m+p)$ real parameters in a locally smooth manner; but it is also
shown there that there exists no {\em globally} smooth
parameterization of $\Ratn_{p,m}$ if $\min(p,m)>1$. The task of
extracting $n(m+p)$ nonredundant conditions out of Wilson's conditions
of stationarity is more delicate, as we shall see.

It has been shown in~\cite{DooGalAbs2008} and stated
in~\cite{GugAntBea2007-u} that, when they have only first-order poles,
the stationary points $\hH(s)$ of the $\calH_2$-norm error function
(i.e., the points where the gradient of $J_{(A,B,C)}$ vanishes) can
be characterized in diagonal canonical form
\begin{equation}  \label{eq:diag}
\hH(s)=\sum_{i=1}^n \frac{\hc_i\hb_i^H}{s-\ha_i}, 
\end{equation}
via tangential interpolation
  conditions which can be formulated as
\begin{align*}
[H^T(s)-\hH^T(s)]\hc_i &= O(s+\ha_i), 
\\ \hb_i^H[H^T(s)-\hH^T(s)]&=O(s+\ha_i),
\\ \hb_i^H[H^T(s)-\hH^T(s)]\hc_i &=O(s+\ha_i)^2.
\end{align*}
Notice that the interpolation points are the negative of the poles of
$\hH(s)$.
These results are, in fact, a consequence
of the relation between the equations of the gradients of the
$\calH_2$-norm error (as derived originally by Wilson
in~\cite{Wil1970}) and tangential interpolation based on Sylvester
equations (as derived
in~\cite{BalGohRod1990},\cite{GalVanDoo2004-JCAM},\cite{GalVanDoo2004-SIMAX}).
Similar conditions can be found in~\cite{BunKubVosWil2007-TR} for the
discrete-time case. 
Observe that the diagonal canonical form~\eqref{eq:diag} uses the
minimal number, $n(m+p)$, of parameters once the $\hb_i$'s or
$\hc_i$'s are normalized to remove the scaling invariance. The
tangential interpolation conditions also impose the correct number of
nonredundant scalar conditions (see Section~\ref{sec:1st}). However,
in view of the result of Byrnes and Falb, the diagonal canonical
form~\eqref{eq:diag}---as well as any other canonical form---cannot
yield a globally smooth one-to-one parameterization of $\Ratn_{p,m}$
when $\min(p,m)>1$. Singularities appear when $\hH(s)$ has
higher-order poles. This is also true for discrete-time systems.

In this
paper, we characterize the stationary points $\hH(s)$ or $\hH(z)$ of
the $\calH_2$-norm error function in Jordan canonical form, i.e.,
\emph{without} the assumption that they have only first-order
poles. The stationarity conditions elegantly generalize to
higher-order tangential interpolation conditions of degree $k_i-1$ (in
the sense of~\cite{GalVanDoo2004-SIMAX}), where $k_i$ is the size of
the $i$th Jordan block. The interpolation points remain the
negative of the poles $\ha_i$ of $\hH(s)$, and the interpolation
directions are polynomial vectors of degree $k_i-1$, built from
the Jordan-form equivalents of $\hb_i$ and $\hc_i$; see
Theorem~\ref{TIsimple3}.  We also show that these tangential
interpolation conditions contain $n(m+p)$ nonredundant scalar
conditions. The result in Theorem~\ref{TIsimple3} has several
precursors: Aigrain and Williams~\cite{AigWil1949} for the SISO case
with simple real poles, Meier and Luenberger~\cite{MeiLue1967} for the
general SISO case (see also the alternative derivation
in~\cite{GugAntBea2007-u}), and~\cite{DooGalAbs2008} for the MIMO case
with simple poles (see also the remark in~\cite{GugAntBea2007-u}).

Since the set of systems with higher-order poles is nowhere dense in
$\Ratn_{p,m}$, the generalization of the stationarity conditions to
higher-order poles seems to be chiefly of theoretical
interest. Nevertheless, we argue that the case of higher-order poles
cannot be simply brushed aside. First, we show on an example that
$\calH_2$-optimal reduced-order models with higher-order poles do
occur. Second, we point out that the Jordan canonical form changes in
a nonsmooth manner at the higher-order poles and that the tangential
interpolation conditions for $\calH_2$-norm stationary points become ill
conditioned around the systems $\hH(s)$ with higher-order
poles. Therefore, insisting on the Jordan canonical form
parameterization of the $\calH_2$-optimal reduced-order model may seriously
affect the sensitivity of any numerical algorithm using such a
parameterization. When the
influence of a nearby higher-order pole becomes problematic, a
possible remedy is to exploit the full
parameterization~\eqref{eq:full-param}.

It should be kept in mind that the above
discussion only concerns stationarity conditions for the
$\calH_2$-norm error function. The stationary points may be local
minima, saddle points, or local maxima of the $\calH_2$-norm
error function. When a descent iteration is employed, convergence to
saddle points and local maxima is not expected to occur. However, the
method can still be trapped in local, nonglobal minima. Such
spurious local minima exist in the $\calH_2$-optimal model reduction
problem, as we show on a simple example. Computing an
$\calH_2$-optimal reduced-order model is thus a tough (obviously
nonconvex) optimization task. Nevertheless, the computed local minima
tend to yield approximations that are considered satisfactory in
practice, hence the interest for interpolation-based fixed-point type
algorithms as revived recently in,
e.g.,~\cite{BeaGug2007,GugAntBea2007-u,Gug2002}.

The paper is organized as follows. After presenting in
  Section~\ref{sec:H2-problem} the necessary background material on
  the $\calH_2$ approximation problem, in Section~\ref{sec:Wilson} we
  recall Wilson's formulas for the gradient of the $\calH_2$-norm
  error function. In Section~\ref{revisit}, Wilson's first-order
  optimality conditions are expressed in a tangential interpolation
  form obtained by representing the reduced-order
  model in Jordan canonical form---thus covering the case of
  higher-order poles in the reduced-order model. The link to
  tangential interpolation by means of projection matrices that solve
  Sylvester equations is discussed in Section~\ref{sec:TI}. The
  importance of dealing with the case of higher-order poles is
  illustrated in Section~\ref{robust}. Section~\ref{secondorder} shows
  on a simple example that the $\calH_2$-optimal model reduction
  problem is a difficult optimization problem, with spurious local
  minimizers in which local optimization algorithms may get
  trapped. An overview of algorithms for solving the $\calH_2$-optimal
  approximation problem is given in Section~\ref{sec:algs}. The
  discrete-time case is covered in Section~\ref{sec:DT}, and
  conclusions are drawn in Section~\ref{sec:conc}.

\section{The $\calH_2$ approximation problem}
\label{sec:H2-problem}

Much of the material in this section is standard and can be found in, e.g.,~\cite{Ant2005}.
Let $E(s)$ be an arbitrary strictly proper transfer function, with realization triple
$(A_e,B_e,C_e)$. If $E(s)$ is unstable, its $\calH_2$-norm is defined to be $\infty$.
Otherwise, its squared $\calH_2$-norm is defined as the trace of a matrix integral~:
\begin{equation}  \label{eq:H2-int}
\|E(s)\|_{\calH_2}^2 := \trace \int_{-\infty}^\infty E(j\omega)^H
E(j\omega) \frac{d\omega}{2\pi}=\trace \int_{-\infty}^\infty E(j\omega)
E(j\omega)^H \frac{d\omega}{2\pi}.
\end{equation}
By Parseval's identity, this can also be expressed using the state space realization as 
\begin{align*}
\|E(s)\|_{\calH_2}^2 &=
 \trace\int_0^\infty[C_e\exp^{A_et}B_e][C_e\exp^{A_et}B_e]^T dt
\\ &=
\trace\int_0^\infty[C_e\exp^{A_et}B_e]^T[C_e\exp^{A_et}B_e] dt.
\end{align*}
This can also be related to an expression involving the gramians $P_e$
and $Q_e$ defined as
\[
P_e:=\int_0^\infty [\exp^{A_et}B_e][\exp^{A_et}B_e]^T dt, \quad
Q_e:=\int_0^\infty[\exp^{A_et}B_e]^T[C_e\exp^{A_et}] dt,
\]
which are also known to be the solutions of the Lyapunov equations
\begin{equation}  \label{eq:Lyap}
A_e P_e + P_eA_e^T + B_eB_e^T =0, \quad
Q_e A_e + A_e^TQ_e + C_e^TC_e =0.
\end{equation}
Using these, it easily follows that the squared $\calH_2$-norm of $E(s)$ can be expressed as
\begin{equation}  \label{trace} \|E(s)\|_{\calH_2}^2 = \trace \; B_e^TQ_eB_e = \trace \; C_eP_eC_e^T. 
\end{equation} 

\begin{rmrk} It is easy to show that if $A_e$ has a single real eigenvalue $\lambda$ 
that tends to zero, i.e., $A_e$ tends to lose its stability~:
$$ A_ex=\lambda x, \quad y^TA_e=\lambda y^T, \quad \lambda \rightarrow 0  
$$
then $P_e$ and $Q_e$ tend to a rank one matrix of infinite norm, since
$$ P_e \rightarrow xx^T \beta/(2\lambda), \quad Q_e \rightarrow yy^T \gamma/(2\lambda), 
\quad \mathrm{where} \quad \beta=y^TB_eB_e^Ty, \quad \gamma=x^TC_e^TC_e^Tx. $$
It then follows that $\calJ \rightarrow \beta\gamma/(2\lambda)$ also becomes infinite.
Similar behavior is also found for complex conjugate pairs of eigenvalues tending to
the imaginary axis. It thus follows that the squared $\calH_2$-norm of $E(s)$ tends
to infinity as soon as $A_e$ looses its stability. This explains why this
norm is typically defined to be infinite when $E(s)$ is unstable.
\end{rmrk}

We now apply this to the error function 
\[
E(s):=H(s)-\hH(s) = C(sI_N-A)^{-1}B - \hC(sI_n-\hA)^{-1}\hB.
\]
A realization of $E(s)$ 
in partitioned form is given by
\begin{equation} \label{AeBeCe} (A_e,B_e,C_e) := \left( 
\begin{bmatrix} A & \\ & \hA \end{bmatrix}, \begin{bmatrix} B \\
\hB \end{bmatrix}, \begin{bmatrix} C & -\hC \end{bmatrix} \right),
\end{equation}
and the Lyapunov equations (\ref{eq:Lyap}) become
\begin{equation}  \label{eq:P_e_error}
P_e :=\begin{bmatrix} P & X
  \\ X^T & \hP \end{bmatrix} , \quad
\begin{bmatrix} A & \\ & \hA \end{bmatrix} \begin{bmatrix} P & X
  \\ X^T & \hP \end{bmatrix} 
+ \begin{bmatrix} P & X
  \\ X^T & \hP \end{bmatrix} \begin{bmatrix} A^T & \\ & \hA^T \end{bmatrix}
+ \begin{bmatrix} B \\ \hB \end{bmatrix} \begin{bmatrix} B^T & \hB^T \end{bmatrix} = 0,
\end{equation}
and
\begin{equation}  \label{eq:Q_e_error}
Q_e :=\begin{bmatrix} Q & Y
  \\ Y^T & \hQ \end{bmatrix} , \quad
\begin{bmatrix} A^T & \\ & \hA^T \end{bmatrix} \begin{bmatrix} Q & Y
  \\ Y^T & \hQ \end{bmatrix} 
+ \begin{bmatrix} Q & Y
  \\ Y^T & \hQ \end{bmatrix} \begin{bmatrix} A & \\ & \hA \end{bmatrix}
+ \begin{bmatrix} C^T \\ -\hC^T \end{bmatrix} \begin{bmatrix} C & -\hC \end{bmatrix} = 0.
\end{equation}

In order to minimize the $\calH_2$-distance
  $\|H(s)-\hH(s)\|_{\calH_2}^2$ of the low-order system
  $\hH(s)=\hC(sI_n-\hA)^{-1}\hB$ to a given the full-order model $H(s)
  = C(sI_N-A)^{-1}B$, we must minimize the function $\calJ_{(A,B,C)}$
  defined by
\begin{subequations}
\label{eq:J}
\begin{equation} \label{eq:J-def}
\calJ_{(A,B,C)}(\hA,\hB,\hC) = \|C(sI_N-A)^{-1}B -
\hC(sI_n-\hA)^{-1}\hB\|_{\calH_2}^2.
\end{equation}
We will frequently omit the subscript in
$\calJ_{(A,B,C)}(\hA,\hB,\hC)$ when the full-order model is clear from
the context. In view of~\eqref{trace}, $\calJ(\hA,\hB,\hC)$ admits the
formulation
\begin{equation} \label{BQB}
\calJ(\hA,\hB,\hC) = \trace\left( \begin{bmatrix} B^T & \hB^T\end{bmatrix}
  \begin{bmatrix} Q & Y \\ Y^T & \hQ \end{bmatrix}
  \begin{bmatrix} B \\ \hB \end{bmatrix} \right) = 
  \trace\left(B^T QB + 2B^TY\hB + \hB^T\hQ\hB\right),  
\end{equation}
where $Q$, $Y$ and $\hQ$ depend on $A$, $\hA$, $C$ and $\hC$ through the Lyapunov equation (\ref{eq:Q_e_error}), 
or equivalently
\begin{equation} \label{CPC}
\calJ(\hA,\hB,\hC) = \trace\left( \begin{bmatrix} C & -\hC\end{bmatrix}
  \begin{bmatrix} P & X \\ X^T & \hP \end{bmatrix}
  \begin{bmatrix} C^T \\ -\hC^T \end{bmatrix} \right) = 
  \trace\left(CPC^T - 2CX\hC^T + \hC\hP\hC^T\right),  
\end{equation}
\end{subequations}
where $P$, $X$ and $\hP$ depend on $A$, $\hA$, $B$ and $\hB$ through the Lyapunov equation 
(\ref{eq:P_e_error}). Note that the terms $B^TQB$ and $CPC^T$ in the above expressions are 
constant, and hence can be discarded in the optimization. 

\begin{rmrk} The Sylvester equations~\eqref{eq:P_e_error}
  and~\eqref{eq:Q_e_error} are nonsingular if and only if the union of
  the spectra of $A$ and $\hA$ does not contain any pair of opposite
  points (see~\cite[Ch.~VI]{Gan59}). In particular, they are
  nonsingular if the transfer functions $H(s)=C(sI_N-A)^{-1}B$ and $\hH(s) =
  \hC(sI_n-\hA)^{-1}\hB$ are stable. In fact, the function
\[
(A,B,C,\hA,\hB,\hC)\mapsto \calJ_{(A,B,C)}(\hA,\hB,\hC)
\]
is smooth around every point where $H(s)$ and $\hH(s)$ are stable. In
particular, when $H(s)$ is stable, the function 
\[
(\hA,\hB,\hC)\mapsto \calJ_{(A,B,C)}(\hA,\hB,\hC)
\]
is smooth around every point where $\hH(s)$ is stable.
\end{rmrk}

\section{Gradients of the squared $\calH_2$-norm error function}
\label{sec:Wilson}

The expansions above can be used to obtain formulas for the gradients
of the squared $\calH_2$-norm error function $\calJ$ versus $\hA$,
$\hB$, and $\hC$.  We define the gradients as follows.

\begin{dfntn}
The gradients of a real-valued function $f(\hA,\hB,\hC)$ of a
  real matrix variables $\hA\in\R^{n\times n}$, $\hB\in\R^{n\times
    m}$, $\hC\in\R^{p\times n}$, are the real matrices $\nabla_\hA f(\hA,\hB,\hC)\in\R^{n\times n}$, $\nabla_\hB f(\hA,\hB,\hC)\in\R^{n\times
    m}$, $\nabla_\hC f(\hA,\hB,\hC)\in\R^{p\times n}$, defined by 
\begin{align*}
[\nabla_\hA f(\hA,\hB,\hC)]_{i,j} &= \frac{\partial}{\partial
  \hA_{i,j}}f(\hA,\hB,\hC), \quad i=1,\ldots,n, \quad j=1,\ldots,n,
\\ [\nabla_\hB f(\hA,\hB,\hC)]_{i,j} &= \frac{\partial}{\partial
  \hB_{i,j}}f(\hA,\hB,\hC), \quad i=1,\ldots,n, \quad j=1,\ldots,m,
\\ [\nabla_\hC f(\hA,\hB,\hC)]_{i,j} &= \frac{\partial}{\partial
  \hC_{i,j}}f(\hA,\hB,\hC), \quad i=1,\ldots,p, \quad j=1,\ldots,n.
\end{align*}
\end{dfntn}
We will write $\nabla_\hA f$ as a compact notation for
$\nabla_\hA f(\hA,\hB,\hC)$ when the argument is clear from the context.

Starting from the characterizations (\ref{eq:P_e_error},\ref{CPC}) and (\ref{eq:Q_e_error},\ref{BQB}) 
of the $\calH_2$ norm, one can derive succinct forms of the gradients.
This theorem is originally due to Wilson~\cite{Wil1970}, but we state here the version derived in
\cite{DooGalAbs2008}, where a proof based on inner products and traces is given.

\begin{thrm} \label{gradientsJ}
The gradients $\nabla_\hA \calJ$, $\nabla_\hB \calJ$ and $\nabla_\hC
\calJ$ of the squared $\calH_2$-norm error $\calJ$~\eqref{eq:J},
where both 
$(A,B,C)$ and $(\hA,\hB,\hC)$
are minimal (i.e., controllable and observable), are given by 
\begin{equation} \label{gradients} 
\nabla_\hA \calJ = 2(\hQ\hP+Y^TX), \quad 
\nabla_\hB \calJ = 2(\hQ\hB+Y^TB), \quad  \nabla_\hC \calJ = 2(\hC \hP-CX), 
\end{equation}
where
\begin{eqnarray} \label{Syl1}A^TY+Y\hA -C^T\hC=0, & \quad & \hA^T\hQ+\hQ\hA +\hC^T\hC=0, \\ 
\label{Syl2}X^TA^T+\hA X^T +\hB B^T=0, & \quad & \hP\hA^T+\hA\hP +\hB\hB^T=0. \end{eqnarray}
\end{thrm}

The gradient forms of Theorem \ref{gradientsJ} allowed us to derive
in~\cite{DooGalAbs2008} a theorem that also provides an important link
to tangential interpolation by projection.

\section{Stationarity conditions in Jordan form}
\label{revisit}

In this section, we revisit Wilson's conditions
  (Theorem~\ref{gradientsJ}) with $\hH(s)$ in Jordan canonical form.
We first consider the continuous-time case and discuss the
discrete-time case in Section~\ref{sec:DT}. 

We will assume that both transfer functions $H(s)$ and $\hH(s)$ have
real minimal (controllable and observable) realizations $(A,B,C)$ and 
$(\hA,\hB,\hC)$. 

\subsection{First-order poles}
\label{sec:1st}

We first assume that all the poles of $\hH(s)$ are distinct (but
possibly complex), which implies that the Jordan canonical form
reduces to a diagonal form.

Since the number of parameters in the full
  parameterization~\eqref{eq:full-param} is not minimal,
the gradient conditions of Theorem \ref{gradientsJ} must be
redundant. This is made
explicit in the theorem below, proved in~\cite{DooGalAbs2008}. 
For this we will need $s_i$, $t_i^H$, 
the (complex) left and right 
eigenvectors of the (real) matrix $\hA$ corresponding to the (complex) eigenvalue $\ha_i$. 
We then have~:
\begin{equation}\label{st} 
\hA s_i=\ha_i s_i, \quad \hC s_i =: \hc_i, \quad t_i^H\hA=\ha_it_i^H,
\quad t_i^H\hB =: \hb_i^H, \quad i=1,\ldots,n,
\end{equation}
and $\hH(s)$ has the partial fraction expansion
\begin{equation} \label{simpletfe}
\hH(s)=\sum_{i=1}^n \frac{\hc_i\hb_i^H}{s-\ha_i}, 
\end{equation}
where $\hb_i\in \C^m$ and $\hc_i\in \C^p$ and where 
$\{(\ha_i,\hb_i, \hc_i): i=1,\ldots,n\}$ is a self-conjugate set. The
  form~\eqref{simpletfe} corresponds to the diagonal canonical form of
  $\hH(s)$, a particular case of the Jordan canonical form when all
  the Jordan blocks have dimension one. It involves the minimal number
  $n(m+p)$ of parameters once normalization conditions are imposed on
  either the $\hb_i$'s or the $\hc_i$'s.

\begin{thrm} \label{TIdistinct}
Let $H(s)=C(sI_N-A)^{-1}B$ and $\hH(s) = \hC(sI_n-\hA)^{-1}\hB$
  be real minimal realizations, and let $\ha_i$, $\hb_i$, $\hc_i$,
  $s_i$, and $t_i$, $i=1,\ldots,n$, be as in~\eqref{st}. Assume
  that $-\ha_i$ is not a pole of $H(s)$, $i=1,\ldots,n$. Then
\begin{eqnarray} \label{TIB} 
\frac{1}{2}(\nabla_\hB \calJ)^Ts_i &=& [H^T(-\ha_i)-\hH^T(-\ha_i)]\hc_i 
\\ \label{TIC} 
\frac{1}{2}t_i^H(\nabla_\hC \calJ)^T &=& \hb_i^H[H^T(-\ha_i)-\hH^T(-\ha_i)] 
\\ \label{TIA} 
\frac{1}{2}t_i^H(\nabla_\hA \calJ)^Ts_i &=& \hb_i^H\frac{d}{ds}\left.[H^T(s)-\hH^T(s)]\right|_{s=-\ha_i}\hc_i 
\\ \label{TIAoff} 
\frac{1}{2}t_i^H(\nabla_\hA \calJ)^Ts_j &=& \frac{1}{2(\ha_i-\ha_j)}[\hb_i^H(\nabla_\hB \calJ)^Ts_j-t_i^H(\nabla_\hC \calJ)^T\hc_j], \; i\neq j,
\end{eqnarray}
where $\calJ$ is the squared $\calH_2$-norm error defined in~\eqref{eq:J}.
\end{thrm}

Let $S:=\begin{bmatrix}s_1 & \ldots & s_n \end{bmatrix}$, then the above theorem shows that the off-diagonal elements of
$S^{-1}(\nabla_\hA \calJ)^TS$ actually depend on $(\nabla_\hB \calJ)^T$ and $(\nabla_\hC \calJ)^T$. Therefore one need
only impose conditions on $\diag(S^{-1}(\nabla_\hA \calJ)^TS)$ and on $(\nabla_\hB \calJ)^T$ and $(\nabla_\hC \calJ)^T$
to characterize stationary points of $\calJ$. 
The following corollary easily follows. It is derived independently
in~\cite{BunKubVosWil2007-TR} for the discrete-time case, and also suggested in~\cite{GugAntBea2007-u}.
\begin{cor} \label{TIsimple}
With the notation and assumptions of Theorem~\ref{TIdistinct},
if $(\nabla_\hB \calJ)^T=0$, $(\nabla_\hC \calJ)^T=0$ and $\diag(S^{-1}(\nabla_\hA \calJ)^TS)=0$, then $\nabla_\hA \calJ=0$
and the following tangential interpolation conditions are satisfied
for all $\ha_i$, $i=1, \ldots, n$~:
\begin{align}
[H^T(-\ha_i)-\hH^T(-\ha_i)]\hc_i &=0, \label{eq:TI-right}
\\  \hb_i^H[H^T(-\ha_i)-\hH^T(-\ha_i)] &=0,  \label{eq:TI-left}
\\ \hb_i^H\frac{d}{ds}\left.[H^T(s)-\hH^T(s)]\right|_{s=-\ha_i}\hc_i
&=0.  \label{eq:TI-der}
\end{align}
\end{cor}
These tangential interpolation conditions contain $n(m+p)$
  nonredundant conditions. To see this, fix $i$ and consider first the
  case where $\ha_i$ is real. The first two
  equations~\eqref{eq:TI-right} and \eqref{eq:TI-left} impose that the
  determinant of $[H^T(-\ha_i)-\hH^T(-\ha_i)]$ vanishes, which
  accounts for one real scalar condition. Next,~\eqref{eq:TI-right}
  and \eqref{eq:TI-left} require that $\hc_i$ and $\hb_i$ belong to
  the kernel of $[H^T(-\ha_i)-\hH^T(-\ha_i)]$, which imposes $p-1$ and
  $m-1$ real scalar conditions. Finally, the last
  equation~\eqref{eq:TI-der} imposes one real scalar condition, for a
  total of $m+p$ conditions corresponding to the fixed $i$. In the
  complex case, we have a pair of complex-conjugate poles $\ha_i$ and
  $\ha_{i+1}$. The constraint $\det[H^T(-\ha_i)-\hH^T(-\ha_i)]=0$
  imposes two real scalar conditions, the first two equations impose
  further $2(p-1)$ and $2(m-1)$ real scalar conditions, and the last
  equation imposes two real scalar conditions, for a total of $2(m+p)$
  real scalar conditions. The equations for $\ha_{i+1}$ impose the
  same conditions since
  equations~\eqref{eq:TI-right}--\eqref{eq:TI-der} are then just the
  complex conjugate ones as for $\ha_i$. The total for $\ha_i$ and
  $\ha_{i+1}$ is thus $2(m+p)$ real scalar conditions. To conclude,
  observe that $i$ ranges from $1$ to $n$, which yields a total of
  $n(m+p)$ real scalar conditions. This matches the number, $n(m+p)$,
  of independent parameters.

The above conditions can also be expressed in terms of the Taylor expansion of $H(s)-\hH(s)$~:
\[ [H^T(s)-\hH^T(s)]\hc_i = O(s+\ha_i), \quad \hb_i^H[H^T(s)-\hH^T(s)]=O(s+\ha_i), \]
\[ \hb_i^H[H^T(s)-\hH^T(s)]\hc_i =O(s+\ha_i)^2.\]
That formulation is in fact easier to extend to higher-order poles. Observe also that we retrieve 
the conditions of Meier and Luenberger~\cite{MeiLue1967} for the 
  single-input single-output (SISO) case since then $\hb_i^H$ and $\hc_i$ are just nonzero scalars
that can be divided out. The above conditions then become the $2n$ conditions
\[  H(-\ha_i)=\hH(-\ha_i), \quad  \frac{d}{ds}\left.H(s)\right|_{s=-\ha_i}=\frac{d}{ds}\left.\hH(s)\right|_{s=-\ha_i}, \quad i=1, \ldots, n.\]

When the transfer function $\hH(s)$ has repeated first-order poles, the results are essentially the same except that there are bases 
$S_i$ and $T_i^H$ of right and left invariant subspaces corresponding to a single eigenvalue $\ha_i$. We then have
\[ \hA S_i=\ha_i S_i , \quad \hC S_i=\hC_i, \quad T_i^H\hA=\ha_i T_i^H, \quad T_i^H\hB =\hB_i^H, \quad T_i^HS_i=I_k.\]
Theorems \ref{TIdistinct} and \ref{TIsimple} still hold but with the vectors $\hc_i$ and $\hb_i^H$ replaced by the matrices $\hC_i$ and $\hB_i^H$.
It may seem that this implies that we then impose more than $n(m+p)$ conditions, but in fact one can choose the individual vectors
of $S_i$ and $T_i^H$ such that the off diagonal elements of $T_i^H(\nabla_\hA \calJ)^TS_i$ are zero. Only its diagonal elements need
then to be constrained to be zero to force the stationarity conditions.

\subsection{Higher-order poles}

Let us now allow $\hH(s)$ to have multiple and higher-order poles. The
main result is given in Theorem~\ref{TIsimple3}, where we show that the
stationary points of the $\calH_2$-norm error function are
characterized by tangential interpolation conditions whose degree
depends on the size of the Jordan blocks of $\hH(s)$. The result
generalizes Corollary~\ref{TIsimple}.

Let $\hH(s)$ then have the
following minimal (controllable and observable) representation
\begin{equation} \label{multiple}
\hH(s)=\sum_{i=1}^\ell \hH_i(s), \quad \hH_i(s):=\hC_i(sI-\hA_i)^{-1}\hB_i^H, \quad \hA_i:=\begin{bmatrix}\ha_i & -1 & \\ & \ha_i  & \ddots \\
 & & \ddots & -1 \\ & & & \ha_i \end{bmatrix}, \end{equation}
where $\hA_i\in \C^{k_i\times k_i}$, $\hB_i^H\in \C^{k_i\times m}$, $\hC_i\in \C^{p\times k_i}$ and where 
$\{(\hA_i,\hB_i^H, \hC_i): i=1,\ldots,\ell\}$ is a self-conjugate set. Notice that this is essentially the partial fraction expansion of $\hH(s)$
and that there may be more than one Jordan block $\hA_i$ associated with the same complex eigenvalue $\ha_i$. The minimality of the representation
implies linear independence of the leading columns in each block $\hB_i$ and of the trailing rows in each block $\hC_i$ that
correspond to the same eigenvalue $\ha_i$, since these blocks appear as subblocks of a minimal realization of $\hH(s)$. 

\medskip

We will need $S_i$, $T_i^H$, the (complex) left and right 
eigenspaces of the (real) matrix $\hA$ corresponding to the (complex) eigenvalue $\ha_i$. Because of 
the expansion (\ref{simpletfe}), we then have~:
\begin{equation}\label{ST} \hA S_i=S_i\hA_i , \quad \hC S_i=\hC_i, \quad T_i^H\hA=\hA_iT_i^H, \quad T_i^H\hB =\hB_i^H, \quad T_i^HS_i=I_k.\end{equation}
Note also that the matrices $S_i$ and $T_i^H$ are not unique. When there is only one Jordan block associated with an
eigenvalue $\ha_i$, its degree of freedom is just a block scaling $S_iD_i$ and $D_i^{-1}T_i^H$ with $D_i\in \C^{k_i\times k_i}$ 
invertible. When there is more than one 
Jordan block associated with $\ha_i$, the degrees of freedom are more involved, but we associate below right 
and left bases $S_i,T_i$ with each individual Jordan block $A_i$.

\medskip

We will also need the following lemmas in preparation for the main theorem.

\begin{lmm} \label{F}
If $-\lambda$ is not an eigenvalue of $A$, the solution of the matrix equation 
\[A^TY+YF-C^TL=0 \quad \mathrm{with} \quad 
F:=\begin{bmatrix}\lambda & -1 & \\ & \lambda  & \ddots \\
 & & \ddots & -1 \\ & & & \lambda \end{bmatrix} \in \C^{k\times k},
\]
with $L:=\begin{bmatrix} \ell_0 & \ell_1 & \ldots & \ell_{k-1} \end{bmatrix}$, 
is given by
\[ Y=\begin{bmatrix} (A^T+\lambda I)^{-1}C^T & (A^T+\lambda I)^{-2}C^T & \ldots & (A^T+\lambda I)^{-k}C^T
\end{bmatrix}
\begin{bmatrix} \ell_0 & \ell_1 & \ldots & \ell_{k-1}\\ & \ell_0 & \ddots & \vdots \\ & & \ddots & \ell_1 \\ 
& & & \ell_0\end{bmatrix}.\]
Moreover, let 
\[  \phi_\lambda(s) :=\begin{bmatrix} 1 & (s+\lambda) & \ldots & (s+\lambda)^{k-1} \end{bmatrix}^T, 
\quad y(s):=Y\phi_\lambda(s),\]
then 
\[ y(s)= (A^T-sI)^{-1}C^TL\phi_\lambda(s) + O(s+\lambda)^k\]
which means that the $i$th column $y_i$ of $Y$ is also the coefficient of $(s+\lambda)^{i-1}$ in the Taylor expansion
of $(A^T-sI)^{-1}C^TL\phi_\lambda(s)$. 
\end{lmm}
\begin{proof} The first part easily follows from $(A^T+\lambda I)y_1=C^T\ell_0$ and  $(A^T+\lambda I)y_i=C^T\ell_{i-1}+y_{i-1}, \; i>1$.
The second part follows from the identity
\[ (A^T-sI)^{-1}C^T = \sum_{i=1}^{\infty}(s+\lambda)^{i-1}(A^T+\lambda I)^{-i}C^T \]
and from the convolution of this formal series with the polynomial vector $L\phi_\lambda(s)$. 
\end{proof}

\medskip

We also give the dual version of this lemma. 

\begin{lmm} \label{FT}
If $-\lambda$ is not an eigenvalue of $A$, the solution of the matrix equation 
\[X^HA^T+FX^H-R^HB^T=0 
\]
with $F\in \C^{k\times k}$ as above and $R:=\begin{bmatrix}
    r_{k-1} & r_{k-2} & \ldots & r_0 \end{bmatrix}$, 
is given by
\[ X^H=
\begin{bmatrix} r_0^H & r_1^H & \ldots & r_{k-1}^H\\ & r_0^H & \ddots & \vdots \\ & & \ddots & r_1^H \\ 
& & & r_0^H \end{bmatrix}
\begin{bmatrix} B^T(A^T+\lambda I)^{-k}\\ \vdots \\ B^T(A^T+\lambda I)^{-2} \\ B^T(A^T+\lambda I)^{-1}
\end{bmatrix}.\]
Moreover, let 
\[  \psi_\lambda(s) :=\begin{bmatrix} (s+\lambda)^{k-1} & \ldots & (s+\lambda) & 1 \end{bmatrix}, 
\quad x^H(s):=\psi_\lambda(s)X^H,\]
then 
\[ x^H(s)= \psi_\lambda(s)R^HB^T(A^T-sI)^{-1} + O(s+\lambda)^k\]
which means that the $i$th row $x_i^H$ of $X^H$ is also the coefficient of $(s+\lambda)^{i-1}$ in the 
Taylor expansion of $\psi_\lambda(s)R^HB^T(A^T-sI)^{-1}$. 
\end{lmm}
\begin{proof} The proof is just the dual of the previous lemma.
\end{proof}

We first obtain an expression for $\nabla_\hB\calJ$ and
  $\nabla_\hC\calJ$ that exploits the
  Jordan canonical form. The result generalizes formulas~\eqref{TIB}
  and~\eqref{TIC} to higher-order poles.

\begin{thrm} \label{HighBC}
Let $H(s)=C(sI_N-A)^{-1}B$ and $\hH(s) = \hC(sI_n-\hA)^{-1}\hB$
  be real minimal realizations, and let $\hA_i$, $\hB_i$, $\hC_i$,
  $S_i$, and $T_i$, $i=1,\ldots,\ell$, describe the Jordan canonical form
  of $\hH(s)$ as in~\eqref{multiple}
  and~\eqref{ST}. Assume
  that $-\ha_i$ is not a pole of $H(s)$, $i=1,\ldots,\ell$. Define
\[ 
\psi_{\ha_i}(s):= \begin{bmatrix} (s+\ha_i)^{k_i-1}\!& \!\ldots\! &
  \!(s+\ha_i)\! & \!1 \end{bmatrix}, \quad
\phi_{\ha_i}(s):= \begin{bmatrix} 1\! & \!(s+\ha_i)\! & \!\ldots\! &
  \!(s+\ha_i)^{k_i-1} \end{bmatrix}^T.
\]
Then we have
\begin{equation} \label{TI2B} 
\frac{1}{2}(\nabla_\hB \calJ)^TS_i\phi_{\ha_i}(s)=[H^T(s)-\hH^T(s)]\hC_i\phi_{\ha_i}(s) +O(s+\ha_i)^{k_i}, 
\end{equation}
\begin{equation} \label{TI2C} 
\frac{1}{2}\psi_{\ha_i}(s)T_i^H(\nabla_\hC \calJ)^T= \psi_{\ha_i}(s)\hB_i^H[H^T(s)-\hH^T(s)] +O(s+\ha_i)^{k_i},
\end{equation}
where $\calJ$ is the squared $\calH_2$-norm error defined in~\eqref{eq:J}.
\end{thrm}
\begin{proof}
Define $Y_i:=Y S_i$, $\hQ_i:=-\hQ S_i$, $X_i:=-XT_i$ and $\hP_i:=-\hP T_i$. Then we have
\[ A^TY_i + Y_i\hA_i =C^T\hC_i, \quad \hA^T\hQ_i + \hQ_i\hA_i =\hC^T \hC_i, \]
\[ X_i^HA^T + \hA_i X_i^H=\hB_i^HB^T, \quad \hP_i^H\hA^T + \hA_i\hP_i^H =\hB_i^H \hB^T. \]
If $-\ha_i$ is not an eigenvalue of $A$ or $\hA$, both $(A^T-sI)^{-1}$ and $(\hA^T-sI)^{-1}$ have Taylor 
expansions in $(s+\ha_i)$. It then follows from Lemmas \ref{F} and \ref{FT} that 
\begin{eqnarray} \label{yq2} 
Y_i\phi_{\ha_i}(s)&=& (A^T - s I)^{-1}C^T\hC_i\phi_{\ha_i}(s)+O(s+\ha_i)^{k_i}, \\ 
\hQ_i\phi_{\ha_i}(s) &=& (\hA^T - s I)^{-1}\hC^T\hC_i\phi_{\ha_i}(s)+O(s+\ha_i)^{k_i}, 
\end{eqnarray}
\begin{eqnarray} \label{xp2} 
\psi_{\ha_i}(s)X_i^H &=& \psi_{\ha_i}(s)\hB_i^H B^T(A^T - s I)^{-1}+O(s+\ha_i)^{k_i}, \\  
\psi_{\ha_i}(s)\hP_i^H &=& \psi_{\ha_i}(s)\hB_i^H \hB^T(\hA^T - s I)^{-1}+O(s+\ha_i)^{k_i}.
\end{eqnarray}
This then yields
\[\frac{1}{2}(\nabla_\hB \calJ)^TS_i\phi_{\ha_i}(s)=(\hB^T\hQ+B^TY)S_i\phi_{\ha_i}(s)=
[H^T(s)-\hH^T(s)]\hC_i\phi_{\ha_i}(s)+O(s+\ha_i)^{k_i},\]
\[\frac{1}{2}\psi_{\ha_i}(s)T_i^H(\nabla_\hC \calJ)^T =\psi_{\ha_i}(s)T_i^H(\hP\hC^T-X^TC^T)=
\psi_{\ha_i}(s)\hB_i^H[H^T(s)-\hH^T(s)]+O(s+\ha_i)^{k_i}.\] 
\end{proof}

\medskip

\begin{rmrk}
The condition that $-\ha_i$ is not a pole of $H(s)$ is satisfied when choosing 
stable interpolation points $\ha_i$, which is typically the case in the algorithms we discuss below. 
\end{rmrk}

The following generalization of the tangential interpolation
  conditions~\eqref{eq:TI-right} and~\eqref{eq:TI-left} immediately
  follows from the previous theorem.
\begin{cor} \label{TIsimple2} With the notation and assumptions
    of Theorem~\ref{HighBC}, if $\nabla_\hB \calJ=0$ and $\nabla_\hC
    \calJ=0$, then the following tangential interpolation conditions
    are satisfied for all $\ha_i, i=1, \ldots, n$~:
\begin{equation}  \label{eq:TI-J-lr}
[H^T(s)-\hH^T(s)]\hc_i(s) = O(s+\ha_i)^{k_i}, \quad 
\hb_i(s)^H[H^T(s)-\hH^T(s)]= O(s+\ha_i)^{k_i},
\end{equation}
where $\hb_i^H(s):=\psi_{\ha_i}(s)\hB_i^H$ and $\hc_i(s):=\hC_i\phi_{\ha_i}(s)$.
\end{cor}

We now turn to the gradient of $\calJ$ versus $\hA$.  We do not have
expressions for $T_i^H (\nabla_\hA\calJ)^TS_j$ that are clean
extensions of~\eqref{TIA} and~\eqref{TIAoff}, however, we do
generalize the two-sided tangential interpolation
condition~\eqref{eq:TI-der} that follows from
$\nabla_\hA\calJ=0$. This yields the following main theorem, which
states the complete generalization of
Corollary~\ref{TIsimple} to higher-order poles, i.e., the
characterization of stationary points by means of tangential
interpolation conditions.

\pacomm{(The theorem on $\nabla_\hA\calJ$ has been removed. It was
  incorrect as stated. If we replace $\hH$ by its derivative and
  $2k_i$ by $k_i$, then the statement seems correct, but it contains
  only one nontrivial scalar formula. This is because
  $T_i^H(\nabla_\hA \calJ)^TS_i$ is upper triangular, as we shall
  see.)

\begin{thrm} \label{HighA}
Let $\hH(s)$ and $H(s)$ be real rational transfer functions. Let $\nabla_\hB \calJ=0$ and $\nabla_\hC \calJ=0$  then for each Jordan block $\hA_i$ in the realization of $\hH(s)$, we have in addition the relation
\begin{equation} \label{TIAk} 
\frac{1}{2}\psi_{\ha_i}(s)T_i^H(\nabla_\hA \calJ)^TS_i\phi_{\ha_i}(s) = \psi_{\ha_i}(s)\hB_i^H[H^T(s)-\hH^T(s)]\hC_i\phi_{\ha_i}(s)
+O(s+\lambda)^{2k_i}. 
\end{equation}
\end{thrm}
\begin{proof}
Let $S_i$, $T_i^H$ be as in Theorem \ref{HighBC}. Define as before
\[\hC_i:=\hC S_i, \; Y_i:=Y S_i, \; \hQ_i:=-\hQ S_i, \; \hB_i^H:=T_i^H\hB, \; X^H_i=-T_i^HX^T, 
\; \hP^H_i:=-T_i^H\hP, \]
then we have
\[ A^TY_i + Y_i\hA_i =C^T\hC_i, \quad \hA^T\hQ_i + \hQ_i\hA_i =\hC^T \hC_i, \]
\[ X_i^HA^T + \hA_i X_i^H=\hB_i^HB^T, \quad \hP_i^H\hA^T + \hA_i\hP_i^H =\hB_i^H \hB^T. \]
>From Theorem \ref{gradientsJ} it follows that $\frac{1}{2}\nabla_\hA \calJ = \hP\hQ+X^TY$.
If we use Lemmas \ref{F}, \ref{FT} and \eqref{convlrs}, we then obtain 
\[ \frac{1}{2}\psi_{\ha_i}(s)T_i^H(\nabla_\hA \calJ)^TS_i\phi_{\ha_i}(s) = \psi_{\ha_i}(s)\hB_i^H [H^T(s)-\hH^T(s)]\hC_i\phi_{\ha_i}(s)+O(s+\ha_i)^{2k_i}. \]
\end{proof}
}

\begin{thrm} \label{TIsimple3}
With the notation and assumptions
    of Theorem~\ref{HighBC},
if $\nabla_\hB \calJ=0$, $\nabla_\hC \calJ=0$ and $\nabla_\hA
\calJ=0$, then the following tangential interpolation conditions are
satisfied for $i=1,\ldots,\ell$:
\begin{align}  
[H^T(s)-\hH^T(s)]\hc_i(s) &= O(s+\ha_i)^{k_i}, \label{eq:TI-J-l}
\\ \hb_i(s)^H[H^T(s)-\hH^T(s)] &= O(s+\ha_i)^{k_i}, \label{eq:TI-J-r}
\\ \hb_i(s)^H[H^T(s)-\hH^T(s)]\hc_i(s) &= O(s+\ha_i)^{2k_i}, \label{eq:TI-J-2s}
\end{align}
where $\hb_i^H(s):=\psi_{\ha_i}(s)\hB_i^H$ and $\hc_i(s):=\hC_i\phi_{\ha_i}(s)$.
\end{thrm}

\begin{proof}
Conditions~\eqref{eq:TI-J-l} and~\eqref{eq:TI-J-r} were obtained in
Corollary~\ref{TIsimple2}. It remains to show that~\eqref{eq:TI-J-2s}
holds. 

We can interpret conditions~\eqref{eq:TI-J-l}--\eqref{eq:TI-J-2s} in terms
of Taylor expansions of the error function $E(s):=H(s)-\hH(s)$. Let
\[ E(s):= \sum_{j=0}^{\infty} E_i(s+\ha_i)^j, \quad \hc_i(s):= \sum_{j=0}^{k_i} l_j(s+\ha_i)^j,\quad \hb_i^H(s):= \sum_{j=0}^{k_i} r_j^H(s+\ha_i)^j,\]
be the Taylor expansions around $s=-\ha_i$ of the rational function $E(s)$ and of the polynomials $\hc_i(s)$
and $\hb_i(s)^H$. Then conditions~\eqref{eq:TI-J-l}--\eqref{eq:TI-J-2s}
are respectively equivalent to
\begin{equation}\label{convl} 
\begin{bmatrix}  E_0^H & E_1^H & \ldots & E_{k_i-1}^H\\ & E_0^H & \ddots & \vdots \\ & & \ddots & E_1^H \\ & & & E_0^H  
\end{bmatrix} \begin{bmatrix} l_0 & l_1 & \ldots & l_{k_i-1}\\ & l_0 & \ddots & \vdots \\
 & & \ddots & l_1 \\ & & & l_0 
\end{bmatrix}=0,\end{equation}
\begin{equation}\label{convr} 
\begin{bmatrix} r_0^H & r_1^H & \ldots & r_{k_i-1}^H\\ & r_0^H & \ddots & \vdots \\ & & \ddots & r_1^H \\ & & & r_0^H 
\end{bmatrix}
\begin{bmatrix}  E_0^H & E_1^H & \ldots & E_{k_i-1}^H\\ & E_0^H & \ddots & \vdots \\ & & \ddots & E_1^H \\ & & & E_0^H  
\end{bmatrix}=0,\end{equation}
and
\begin{equation}\label{convlr}
\begin{bmatrix} r_0^H & r_1^H & \ldots & r_{2k_i-1}^H\\ & r_0^H & \ddots & \vdots \\ & & \ddots & r_1^H \\ & & & r_0^H 
\end{bmatrix}
\begin{bmatrix}  E_0^H & E_1^H & \ldots & E_{2k_i-1}^H\\ & E_0^H & \ddots & \vdots \\ & & \ddots & E_1^H \\ & & & E_0^H  
\end{bmatrix}\begin{bmatrix} l_0 & l_1 & \ldots & l_{2k_i-1}\\ & l_0 & \ddots & \vdots \\
 & & \ddots & l_1 \\ & & & l_0 
\end{bmatrix}=0.
\end{equation}
The condition that the first $k_i$ or $2k_i$ terms of the Taylor expansion vanish is indeed equivalent to the
fact that the above partial convolutions are zero. 
We know that~\eqref{convl} and~\eqref{convr} hold,
since~\eqref{eq:TI-J-l} and~\eqref{eq:TI-J-r} hold; it remains to
show~\eqref{convlr} to conclude the proof. 

We will need the identity
\begin{eqnarray} \nonumber
& \begin{bmatrix}  E_{k_i}^H & \ldots & E_{2k_i-1}^H\\ 
\vdots & \ddots & \vdots \\ E_1^H & \ldots & E_{k_i}^H
\end{bmatrix} = \\  \nonumber &
\begin{bmatrix} B^T(A^T+\ha_i I)^{-k_i}\\ \vdots  \\ B^T(A^T+\ha_i I)^{-1}
\end{bmatrix}
 \begin{bmatrix} (A^T+\ha_i I)^{-1}C^T &  \ldots & (A^T+\ha_i I)^{-k_i}C^T
\end{bmatrix} \\
&- \begin{bmatrix} \hB^T(\hA^T+\ha_i I)^{-k_i}\\ \vdots  \\ \hB^T(\hA^T+\ha_i I)^{-1}
\end{bmatrix}
 \begin{bmatrix} (\hA^T+\ha_i I)^{-1}\hC^T &  \ldots & (\hA^T+\ha_i I)^{-k_i}\hC^T
\end{bmatrix},   \label{eq:E-factorized}
\end{eqnarray}
which holds since
\[E_{f+g-1}^H=B^T(A^T+\lambda_i I)^{-f}(A^T+\lambda_i I)^{-g}C^T-
\hB^T(\hA^T+\lambda_i I)^{-f}(\hA^T+\lambda_i I)^{-g}\hC^T. \]

Define
\begin{equation}  \label{eq:Yi-etc}
Y_i:=Y S_i, \; \hQ_i:=-\hQ S_i, \; X^H_i=-T_i^HX^T, 
\; \hP^H_i:=-T_i^H\hP. 
\end{equation}
Using Wilson's formulas (Theorem~\ref{gradientsJ}) for the first
equality, Lemmas~\ref{F} and~\ref{FT} for the second one,
and the identity~\eqref{eq:E-factorized} for the third,
we have
\begin{align}
&T_i^H(\nabla_\hA\calJ)^TS_i 
\\ =&
\hP_i^H\hQ_i - X_i^HY_i
\nonumber
\\ 
=& 
\begin{bmatrix} r_0^H & r_1^H & \ldots & r_{k_i-1}^H\\ & r_0^H & \ddots & \vdots \\ & & \ddots & r_1^H \\ & & & r_0^H 
\end{bmatrix}
\begin{bmatrix} B^T(A^T+\ha_i I)^{-k_i}\\ \vdots  \\ B^T(A^T+\ha_i I)^{-1}
\end{bmatrix}
\\ &
 \begin{bmatrix} (A^T+\ha_i I)^{-1}C^T &  \ldots & (A^T+\ha_i I)^{-k_i}C^T
\end{bmatrix}
\begin{bmatrix} l_0 & l_1 & \ldots & l_{k_i-1}\\ & l_0 & \ddots & \vdots \\
 & & \ddots & l_1 \\ & & & l_0 
\end{bmatrix}
\nonumber
\\
&- 
\begin{bmatrix} r_0^H & r_1^H & \ldots & r_{k_i-1}^H\\ & r_0^H & \ddots & \vdots \\ & & \ddots & r_1^H \\ & & & r_0^H 
\end{bmatrix}
\begin{bmatrix} \hB^T(\hA^T+\ha_i I)^{-k_i}\\ \vdots  \\ \hB^T(\hA^T+\ha_i I)^{-1}
\end{bmatrix}
\nonumber
\\ &
 \begin{bmatrix} (\hA^T+\ha_i I)^{-1}\hC^T &  \ldots & (\hA^T+\ha_i I)^{-k_i}\hC^T
\end{bmatrix}
\begin{bmatrix} l_0 & l_1 & \ldots & l_{k_i-1}\\ & l_0 & \ddots & \vdots \\
 & & \ddots & l_1 \\ & & & l_0 
\end{bmatrix}
\nonumber
\\
=& 
- \begin{bmatrix} r_0^H & r_1^H & \ldots & r_{k_i-1}^H\\ & r_0^H & \ddots & \vdots \\ & & \ddots & r_1^H \\ & & & r_0^H 
\end{bmatrix}
\begin{bmatrix}  E_{k_i}^H & \ldots & E_{2k_i-1}^H\\ 
\vdots & \ddots & \vdots \\ E_1^H & \ldots & E_{k_i}^H
\end{bmatrix}
\begin{bmatrix} l_0 & l_1 & \ldots & l_{k_i-1}\\ & l_0 & \ddots & \vdots \\
 & & \ddots & l_1 \\ & & & l_0 
\end{bmatrix}.
\label{convlrs-grad}
\end{align}
We are now ready to show~\eqref{convlr}. Since~\eqref{convl}
and~\eqref{convr} hold, the left-hand side of~\eqref{convlr} satisfies
\begin{align}
&
\begin{bmatrix} r_0^H & r_1^H & \ldots & r_{2k_i-1}^H\\ & r_0^H & \ddots & \vdots \\ & & \ddots & r_1^H \\ & & & r_0^H 
\end{bmatrix}
\begin{bmatrix}  E_0^H & E_1^H & \ldots & E_{2k_i-1}^H\\ & E_0^H & \ddots & \vdots \\ & & \ddots & E_1^H \\ & & & E_0^H  
\end{bmatrix}\begin{bmatrix} l_0 & l_1 & \ldots & l_{2k_i-1}\\ & l_0 & \ddots & \vdots \\
 & & \ddots & l_1 \\ & & & l_0 
\end{bmatrix}
\nonumber
\\ =&
\begin{bmatrix}
0 & 
\begin{bmatrix} r_0^H & r_1^H & \ldots & r_{k_i-1}^H\\ & r_0^H & \ddots & \vdots \\ & & \ddots & r_1^H \\ & & & r_0^H 
\end{bmatrix}
\begin{bmatrix}  E_{k_i}^H & \ldots & E_{2k_i-1}^H\\ 
\vdots & \ddots & \vdots \\ E_1^H & \ldots & E_{k_i}^H
\end{bmatrix}
\begin{bmatrix} l_0 & l_1 & \ldots & l_{k_i-1}\\ & l_0 & \ddots & \vdots \\
 & & \ddots & l_1 \\ & & & l_0 
\end{bmatrix}
\\
0 & 0
\end{bmatrix}
\nonumber
\\ =&
\begin{bmatrix}
0 & -T_i^H(\nabla_\hA\calJ)^TS_i \\ 0 & 0 
\end{bmatrix},
\label{convlrs}
\end{align}
where the first equality follows from a careful blockwise inspection, and
the second equality uses~\eqref{convlrs-grad}.
Since $\nabla_\hA\calJ=0$, it follows that~\eqref{convlr} holds, and
thus~\eqref{eq:TI-J-2s} holds.
\end{proof}

\subsection{Number of parameters and conditions}
\label{sec:count}

\pacomm{This whole section is new. It is also particularly shaky. My
  concern is that once the $k_i$'s are known, the number of degrees of
  freedom is reduced by $\sum_{i=1}^\ell (k_i-1)$. This does not show
  up in the development below: we still count $n(m+p)$ nonredundant
  conditions. -PA, 21 JUL 2008}

In this subsection, we show that the tangential interpolation
conditions obtained in
Theorem~\ref{TIsimple3}---i.e.,~\eqref{eq:TI-J-l}--\eqref{eq:TI-J-2s}---impose
the correct number, $n(m+p)$, of nonredundant scalar conditions.

To this end, fix $i$ and consider the Jordan block of size $k_i$
associated to $\lambda_i$.  The tangential interpolation conditions
are equivalent to~\eqref{convl}--\eqref{convlr}. Both~\eqref{convl}
and~\eqref{convr} agree on imposing that
\[
\begin{bmatrix}  E_0^H & E_1^H & \ldots & E_{k_i-1}^H\\ & E_0^H & \ddots & \vdots \\ & & \ddots & E_1^H \\ & & & E_0^H  
\end{bmatrix}
\]
has a kernel of dimension $k_i$. Indeed, the fact that the realization
is observable imposes that $\ell_0\neq0$, and thus the $k_i$ columns of
\[
\begin{bmatrix} l_0 & l_1 & \ldots & l_{k_i-1}\\ & l_0 & \ddots & \vdots \\
 & & \ddots & l_1 \\ & & & l_0 
\end{bmatrix}
\]
are linearly independent. This counts for $k_i$ conditions.  Next,
in~\eqref{convl}, the equations in columns $1$ to $k_i-1$ are
redundant with the equations in column $k_i$. There are thus $k_ip$
conditions, but the left-hand matrix is known to have a kernel of
dimension $k_i$; this reduces the number of nontrivial conditions to
$k_ip-k_i$. The same reasoning on~\eqref{convr} leads to $k_im-k_i$
conditions. Finally, once~\eqref{convl} and~\eqref{convr} hold, the
two-sided condition~\eqref{eq:TI-J-2s}, equivalent to~\eqref{convlr},
imposes $k_i$ additional conditions. This is because the left-hand
side of~\eqref{convlr} reduces to~\eqref{convlrs}, a Toeplitz matrix
with only $k_i$ nonzero diagonals. In total for $i$, we have
$k_i(m+p)$ nonredundant conditions. The overall total is thus
$\sum_{i=1}^\ell k_i(m+p) = n(m+p)$, which is the dimension of
$\Ratn_{p,m}$.

\section{Relation with tangential interpolation by projection}
\label{sec:TI}

The gradient forms of Theorem \ref{gradientsJ} yields the following
theorem (proved in~\cite{DooGalAbs2008}) that provides an important link
to tangential interpolation by projection.
\begin{thrm} \label{projecthat} At every stationary point of
  $\calJ$~\eqref{eq:J} where $\hP$ and $\hQ$ are invertible, we have
  the following identities
\begin{equation} \label{WV} \hA= W^TAV, \quad \hB= W^TB, \quad \hC= CV, \quad W^TV=I_n \end{equation}
where $W:=-Y\hQ^{-1}$, $V:=X\hP^{-1}$ and $X$, $Y$, $\hP$ and $\hQ$ satisfy the Sylvester equations (\ref{Syl1},\ref{Syl2}).
\end{thrm}

If we rewrite the above theorem as a projection problem, then we are constructing a projector
$\Pi:=VW^T$ (implying $W^TV=I_n$) where $V$ and $W$ are given by the following (transposed) Sylvester equations
\begin{equation} \label{eq:VW} 
(\hQ W^T)A+\hA^T(\hQ W^T) +\hC^T C=0, \quad  A(V\hP)+(V\hP)\hA^T +B \hB^T=0.
\end{equation}
Note that $\hP$ and $\hQ$ can
be interpreted as normalizations to ensure that $W^TV=I_n$.

Rewriting the Sylvester equations~\eqref{eq:VW} as 
\begin{subequations}  \label{eq:Sylv-VW}
\begin{gather}
W^TA + (\hQ^{-1}\hA\hQ)W^T + (\hC\hQ^{-1})C = 0, \\
AV + V(\hP\hA^T\hP^{-1}) + B(\hB^T\hP^{-1}) = 0,
\end{gather}
\end{subequations}
shows the relation with the tangential interpolation described in 
\cite{GalVanDoo2004-SIMAX}. There it is shown that when solving two Sylvester 
equations for the unknowns $W,V\in \R^{N\times n}$
\begin{gather}
W^T A - \Sigma_\mu^T W^T + L^T C = 0,  \label{eq:W-Sylv}
\\  AV -V \Sigma_\sigma + B R = 0,  \label{eq:V-Sylv}
\end{gather}
and constructing the reduced-order model (of degree $n$) as follows
\begin{equation} \label{project}
(\hA,\hB,\hC):= ((W^TV)^{-1}W^TAV,(W^TV)^{-1}W^TB, CV),
\end{equation}
amounts to a tangential interpolation problem (provided the matrix $W^TV$ is invertible). 
The ``interpolation conditions''  $\left(\Sigma_\sigma,R \right)$
and  $\left(\Sigma_\mu,L \right)$ (where $\Sigma_\mu,\Sigma_\sigma\in \R^{n\times n}$,
$R\in \R^{m\times n}$ and $L\in \R^{p\times n}$) are known to uniquely determine the projected 
system $(\hA,\hB,\hC)$~\cite{GalVanDoo2004-SIMAX}. 
Moreover, they reproduce exactly the 
conditions derived in the previous section since they can be expressed in another coordinate system 
by applying invertible transformations of the type $\left(Q^{-1}\Sigma_\sigma Q, RQ \right)$ and  $\left(P^{-1}\Sigma_\mu P, LP \right)$ to the interpolation conditions. This yields transformed 
matrices $VP$ and $WQ$ but does not affect the transfer function of the reduced-order model 
$(\hA,\hB,\hC)$ (see~\cite{GalVanDoo2004-SIMAX} for more details).
The novelty of the derivation in this paper is the case of
higher-order poles: the tangential interpolation conditions in
  Theorem~\ref{TIsimple3} contain fewer redundant equations than those
  that would follow from~\cite{GalVanDoo2004-SIMAX}.

\section{First-order versus higher-order poles}
\label{robust}

In this section we show that $\calH_2$-optimal reduced-order models with repeated poles can indeed occur and that in their neighborhood
one can expect the tangential interpolation approach to have serious numerical difficulties. We start with a lemma 
that will allow us to demonstrate this.

\begin{lmm} \label{exist}
A stable $n$-th degree transfer function $\hat H(s)=\hC(sI_n-\hA)^{-1}\hB$ is a stationary point of the error function
$\|\hH(s)-H(s)\|_{\calH_2}$ if and only if $H(s)$ can be realized as follows
\begin{equation} \label{r1} A= \left[ \begin{array}{cc} \hA & A_{12}\\ A_{21}& A_{22} \end{array} \right], \quad
B= \left[ \begin{array}{c} \hB \\ B_{2} \end{array}\right], \quad
C= \left[ \begin{array}{cc} \hC & C_{2} \end{array}\right], 
\end{equation}  where moreover
\begin{equation} \label{r2} \hA \hP + \hP \hA^T +\hB\hB^T =0, \quad  A_{21} \hP + B_2\hB^T=0,\end{equation}
\begin{equation} \label{r3} \hQ \hA + \hA^T\hQ +\hC^T\hC =0, \quad \hQ A_{12}+ \hC^TC_{2}=0. \end{equation}
\end{lmm}
\begin{proof} The proof follows from the stationarity conditions in Theorem
  \ref{gradientsJ}. 
The ``if'' part is direct: the stationarity conditions
  hold with $X=\left[\begin{smallmatrix} \hP \\ 0 \end{smallmatrix}\right]$
  and $Y=-\left[\begin{smallmatrix} \hQ \\ 0 \end{smallmatrix}\right]$. 
For
  the ``only if'' part, the assumption that $\hH(s)$ is stable
and of degree $n$, guarantees that the matrices $\hP$ and $\hQ$ exist and are invertible. Using $Y^TX=-\hP\hQ$ one can then 
always choose a coordinate system for the realization of $H(s)$ in which  
\[ X=\left[ \begin{array}{cc} \hP \\ 0 \end{array} \right], \quad Y = - \left[ \begin{array}{cc} \hQ \\ 0 \end{array} \right]\]
and hence
\[ W=X\hP^{-1}=\left[ \begin{array}{cc} I_n \\ 0 \end{array} \right], \quad V= -Y\hQ^{-1}=  \left[ \begin{array}{cc} I_n \\ 0 \end{array} \right].\]
Therefore we have $A_{11}=\hA, \; B_{1}=\hB, \; C_{1}=\hC$. 
\end{proof}

\medskip

This special coordinate system can be used to construct a transfer function $H(s)$ for which a {\em given} $\hH(s)$ is the best 
$\calH_2$ norm approximation of $H(s)$. 

\begin{thrm} \label{opt}
Let $\hat H(s)=\hC(sI_n-\hA)^{-1}\hB$ be a given stable $n$-th degree transfer function, then there always exists a 
stable $N$-th degree transfer function $H(s)=C(sI_N-A)^{-1}B$ with $N>n$, for which $\hH(s)$ is a stationary
point of the $\calH_2$ error function.
\end{thrm}
\begin{proof} It suffices to construct $\hP$ and $\hQ$ satisfying the Lyapunov equations in \eqref{r2} and \eqref{r3}, and then
choose $A_{21}=-B_2\hB^T\hP^{-1}$ and $A_{12}=-\hQ^{-1}\hC^TC_2$ to satisfy the conditions of Lemma \ref{exist}. Notice that 
this always has a solution since $\hP$ and $\hQ$ are invertible because $\hH(s)$ is stable and minimal. In order to guarantee that 
$H(s)$ is also stable, one needs to choose the remaining degrees of freedom, i.e. $A_{22}$, $B_2$ and $C_2$
to satisfy this condition. This can be achieved in several ways, but the simplest one is to choose $A_{22}$ stable, and the matrices $B_2$ and $C_2$ sufficiently small. The matrices $A_{21}=-B_2\hB^T\hP^{-1}$ and $A_{12}=-\hQ^{-1}\hC^TC_2$ will then also be small, and $A$ will then be essentially block diagonal and hence stable.
\end{proof}

The above theorem does not show that the constructed stationary point
is also a local minimum, but the following example shows that this is
not too difficult to construct. Choose $\hH(s)=1/(s-a)^2$ with $a=-1$ and a realization
\[ \hA= \left[ \begin{smallmatrix} a & 1 \\ 0 & a  \end{smallmatrix}\right],
\hB= \left[ \begin{smallmatrix} 0 \\ 1  \end{smallmatrix}\right], 
\hC= \left[ \begin{smallmatrix} 1 & 0 \end{smallmatrix}\right] \]
then the transfer function $H(s)=(0.25 s^2 - 0.5 s + 9.25)/(s^3 + 7 s^2 + 19 s + 9)$ with realization
\[ A= \left[ \begin{smallmatrix} a & 1 & d \\ 0 & a & e \\ e & d & f \end{smallmatrix}\right],
B= \left[ \begin{smallmatrix} 0 \\ 1 \\ g \end{smallmatrix}\right], 
C= \left[ \begin{smallmatrix} 1 & 0 & g \end{smallmatrix}\right] \]
with $f=-5$, $g=.5$, $d=4ag$, $e=4a^2g$, is stable and satisfies the stationarity conditions of
Lemma \ref{exist}. Moreover, 1000 random perturbations of the stationary point $\hH(s)$ show that
this is clearly a local minimum of the error function $\|H-\hH\|_{\calH_2}$.

This example shows that if we aim for an $\calH_2$-optimal
reduced-order model $\hH(s)$ with multiple poles, the model reduction
technique that restricts itself to first-order poles will not be able
to produce that solution. However, what happens if we perturb
$H(s)$ or $\hH(s)$?  What can we say about the mapping from one to
the other? This is addressed in the following theorem, which
  shows that if $\hH(s)$ is a stationary point of the
  $\calH_2$-distance to $H(s)$, then every sufficiently nearby
  transfer function $\hH_\Delta(s)$ is a stationary point of a nearby
  system $H_\Delta(s)$.

\begin{thrm} \label{surjective} Let $\hat H(s)=\hC(sI_n-\hA)^{-1}\hB$
  and $H(s)=C(sI_N-A)^{-1}B$ be stable and minimal transfer functions
  such that $\hH(s)$ is a stationary point (resp., nondegenerate
    local minimum) of the error function $\|H(s)-\hH(s)\|_{\calH_2}$.
  Then, for every neighborhood $\calU$ of $H(s)$ in
    $\Ratn_{p,m}$, there exists a neighborhood $\hat{\calU}$ of
    $\hH(s)$ in $\RatN_{p,m}$ such that, for all
    $\hH_\Delta(s)\in\hat{\calU}$, there exists $H_\Delta(s)\in\calU$
    for which $\hH_\Delta(s)$ is a stationary point (resp.,
    nondegenerate local minimum) of the $\calH_2$-distance to
    $H_\Delta(s)$. 
\end{thrm}
\begin{proof} The proof consists of constructing a continuous
    mapping $\psi$ from a neighborhood $\calV$ of $\hH(s)$ in
    $\Ratn_{p,m}$ into $\RatN_{p,m}$ such that $\hH_\Delta(s)$ is a
    stationary point of the $\calH_2$-distance to
    $\psi(\hH_\Delta(s))$ for all $\hH_\Delta(s)$ in $\calV$.
We use Lemma \ref{exist} to do this. Let $(\hA_\Delta,\hB_\Delta,\hC_\Delta)$ be a nearby realization of the nearby system
$\hH_\Delta(s)$. The solution $\hP_\Delta$ and $\hQ_\Delta$ of the perturbed Lyapunov equations 
in \eqref{r2} and \eqref{r3}, will be close to $\hP$ and $\hQ$ by continuity of the solution of a non-singular system of equations.
For the same reason we can construct nearby solutions ${A_{21}}_\Delta=-B_2\hB^T_\Delta \hP_\Delta^{-1}$ and ${A_{12}}_\Delta=-\hQ^{-1}_\Delta\hC^T_\Delta C_2$ 
to finally yield a realization
\[ A_\Delta= \left[ \begin{array}{cc} \hA_\Delta & {A_{12}}_\Delta\\ {A_{21}}_\Delta& A_{22} \end{array} \right], \quad
B_\Delta= \left[ \begin{array}{c} \hB_\Delta \\ B_{2} \end{array}\right], \quad
C_\Delta= \left[ \begin{array}{cc} \hC_\Delta & C_{2} \end{array}\right], 
\] for a transfer function $H_\Delta(s)=:\psi(\hH_\Delta(s))$ which is
close to $H(s)$ and satisfies the conditions of Lemma \ref{exist}.
Since, in view of its expression~\eqref{eq:H2-int}, the $\calH_2$-norm
error function is locally smooth in terms of the coefficients
of system parameters of $H(s)$ and $\hH(s)$, every stationary point
that is a nondegenerate local minimum remains a local minimum for
sufficiently small perturbations.  The proof therefore applies to such
points.
\end{proof} 

This theorem implies that the set of full-order models $H(s)$ that
have $\calH_2$-stationary reduced-order models with only simple poles,
is open and dense in $\RatN_{p,m}$. This follows from the following
reasoning. From the continuity of the mapping from $H(s)$ to $\hH(s)$
and from the fact that the set of systems with only simple poles is
open, it follows that, around a system $H(s)$ with reduced-order
models with only simple poles, there is an neighborhood of systems
with reduced-order models with only simple poles. If $H(s)$ has a
reduced-order model $\hH(s)$ with multiple poles, then, because the
``reduction'' map is an open map and the set of systems with only
simple poles has an empty interior, it follows that any neighborhood
of $H(s)$ contains a full-order model with a reduced-order model with
only simple poles.  One could conclude from this that one need only
consider first-order interpolation techniques, but, when one
approaches a system for which the target function $\hH(s)$ has
multiple poles, the interpolation conditions change in a non-smooth
manner in its neighborhood. The first-order conditions will become
linearly dependent and they will no longer define the reduced-order
model uniquely.  This is obvious in the SISO case. In the MIMO case,
observe that the tangential interpolation
conditions~\eqref{eq:TI-right} involve the interpolation direction
$\hc_i = \hC s_i$, where $s_i$ is the eigenvector of $\hA$ related to
$\ha_i$; if $\ha_i$ and $\ha_{i+1}$ coalesce to form a nontrivial
Jordan block, then the eigenvectors $s_i$ and $s_{i+1}$ merge
(see~\cite{Wil65}) and hence the tangential interpolation directions
merge, too.  This implies that the systems of equations that one
solves become ill-conditioned in the neighborhood of a point where the
solution has higher-order poles.  The same ill-condioned behavior can
be expected for any target system $\hH(s)$ which has no higher-order
poles but is near a system with higher-order poles.

\section{First- and complex second-order approximation}
\label{secondorder}

In this section we consider how the error function changes with the interpolation conditions. 
In order to analyze this, we look at first- and second-order
approximations only, i.e., approximation by systems with one
  real pole or two complex conjugate poles.
If we are looking for a (real) first-order approximation
\[ \hH(s)=cb^T/(s-\lambda) \]
then according to the formulas of Section \ref{revisit}, it should satisfy the following properties at every stationary point of $\calJ$~:
\[ H^T(-\lambda)c=-b\frac{c^Tc}{2\lambda}, \quad b^TH^T(-\lambda)=-c^T\frac{b^Tb}{2\lambda}, \quad 
b^T\frac{d}{ds}H^T(s)c|_{s=-\lambda}=-\frac{b^Tbc^Tc}{4\lambda^2}.\]
If we are looking for a second-order approximation with complex conjugate poles
\[ \hH(s)=cb^H/(s-\lambda)+\overline{c}\overline{b}^H/(s-\overline{\lambda}) \]
then it should satisfy the following properties at every stationary point of $\calJ$~:
\[ H^T(-\lambda)c=-b\frac{c^Hc}{2\lambda}, \quad b^HH^T(-\lambda)=-c^H \frac{b^Hb}{2\lambda}, \quad 
b^H\frac{d}{ds}H^T(s)c|_{s=-\lambda}=-\frac{b^Hbc^Hc}{4\lambda^2}.\]
In both cases, the first two equations express that for every interpolation point $-\lambda$ 
(real or complex) one should choose left and right singular vectors of $H^T(-\lambda)$ as 
tangential interpolation directions $b$ and $c$ for constructing the first- and second-order 
section. The third equation (combined with the two previous ones) expresses that the interpolation 
point is a stationary point of the error function
versus $\lambda$. 

\medskip

If we keep the interpolation point as a parameter, we can plot the error function versus $-\lambda$,
but where $b$ and $c$ are chosen optimal for that interpolation point.
In other words, the optimal approximation $\hH(s)$ is then completely defined by the interpolation point
$-\lambda$. We can therefore have a look at the function we need to optimize by plotting
the error function $\|H(s)-\hH(s)\|_{\calH_2}^2$ as a function of $\lambda$.
It follows from the optimality conditions on $b$ or $c$ that $\|H(s)-\hH(s)\|_{\calH_2}^2=\|H(s)\|_{\calH_2}^2-\|\hH(s)\|_{\calH_2}^2$.
Indeed, let $\nabla_\hB \calJ=Y^TB+\hQ\hB=0$ then
\begin{eqnarray*} 
\|H(s)-\hH(s)\|_{\calH_2}^2 &=& \trace\left(B^T QB + B^TY\hB + \hB^T
  Y^TB + \hB^T\hQ\hB\right)
\\ &=& \trace\left(B^T QB\right) - 
\trace\left(\hB^T\hQ\hB\right) \\ &=& \|H(s)\|_{\calH_2}^2-\|\hH(s)\|_{\calH_2}^2.\end{eqnarray*}
The development for $\nabla_\hC \calJ=CX-\hC\hP=0$ is essentially the same. In the real case we then have 
\[\|\hH(s)\|_{\calH_2}^2=b^T\hH^T(-\lambda)c=\frac{b^Tbc^Tc}{-2\lambda}\]
which implies $\|\hH(s)\|_{\calH_2}^2=\frac{\sigma^2(H(-\lambda))}{-2\lambda}$ because of the above 
formulas. This indicates that we need to choose the vectors $b$ and $c$ corresponding to the largest
singular value of $H(\lambda)$. In the complex case we have 
\[\|\hH(s)\|_{\calH_2}^2=2\Re \left(b^H\hH^T(-\lambda)c+b^T\hH^T(-\overline{\lambda})\overline{c}\right)\]
and the same conclusion follows after some manipulation.

\medskip

In Figure~\ref{fig2} we show this function for a MIMO example with $m=p=2$ and $N=20$, for which the 
optimum is reached at a pair of complex conjugate interpolation points. 
Subplot 1 shows the poles of $H(s)$ (blue crosses) and the poles of
the $\calH_2$-optimal reduced-order model $\hH(s)$ (black circles). 
Subplots 2 and 3 show the log of the $\calH_2$ norm of the error as a
function of the interpolation point $-\lambda$ (both in contour and in
3D view). 
Subplot 4 shows the frequency response norms $\sigma_{\max}(G(j\omega))$, where $G(s)$ is the system $H(s)$, the optimal second-order approximation $\hH(s)$ and the error $H(s)-\hH(s)$.
This system was generated randomly, but the function is not so simple to optimize. It is clearly not convex and there are several basins of 
attraction to local minima that are not optimal. One often recommends to start with the poles closest to
the $j\omega$ axis as interpolation points (or the largest peaks in the frequency response), but for this example that would converge to local minima, as one can see from the $\calH_2$ error plot.
\begin{figure}[th]
\caption{Second-order approximation of MIMO case}\label{fig2}
\includegraphics[width=1.1\textwidth]{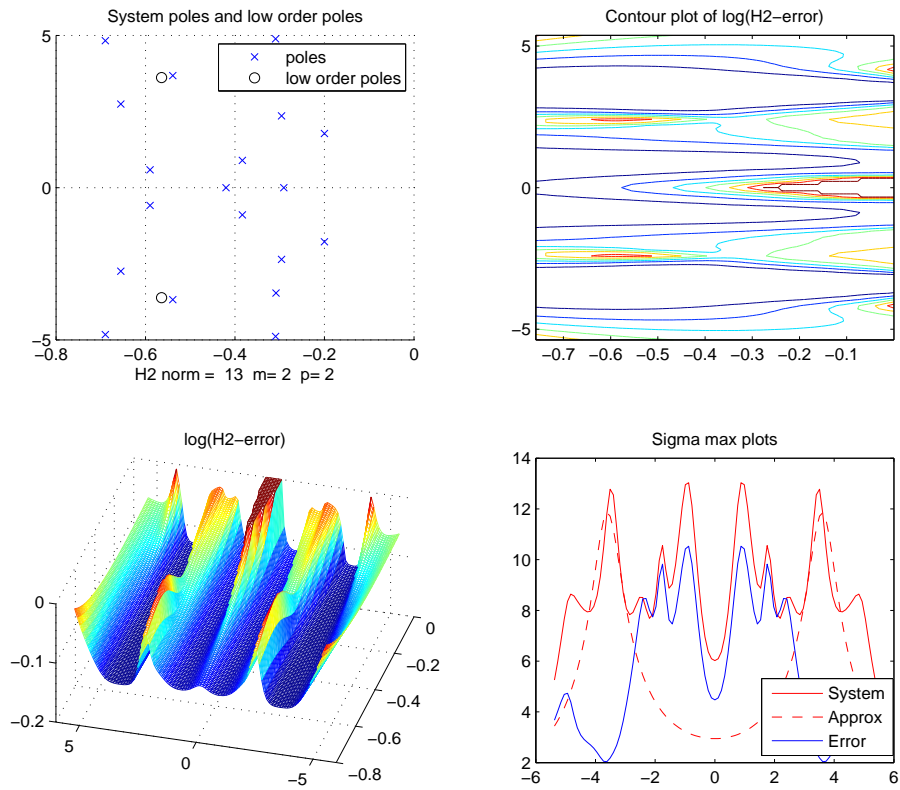}
\end{figure}

\section{Algorithms for solving the interpolation problem}
\label{sec:algs}

One can view (\ref{Syl1},\ref{Syl2}) 
and (\ref{WV}) as two coupled systems of equations
\[(X,Y,\hP,\hQ)=F(\hA,\hB,\hC) \quad \mathrm{and} \quad (\hA,\hB,\hC)=G(X,Y,\hP,\hQ) \]
for which we have a fixed point $(\hA,\hB,\hC)=G(F(\hA,\hB,\hC))$ at every stationary point of 
$\calJ(\hA,\hB,\hC)$. This automatically suggests an iterative procedure
\[(X,Y,\hP,\hQ)_{i+1}=F(\hA,\hB,\hC)_{i+1}, \quad
(\hA,\hB,\hC)_{i+1}=G(X,Y,\hP,\hQ)_{i}, \] which is expected to
converge to a nearby fixed point. This is essentially the idea behind
existing algorithms using Sylvester equations in their iterations
(see~\cite{Ant2005}).  Specifically, this is the idea behind the IRKA
algorithm of~\cite{GugAntBea2007-u}, except that one has to adapt the
formulas to make sure that the matrices $V$ and $W$ satisfy
$W^TV=I_n$.  Another approach would be to use the gradients (or the
interpolation conditions of Theorem \ref{TIdistinct}) to develop
descent methods or even Newton-like methods, as was done for the SISO
case in~\cite{GugAntBea2007-u}. Quasi-Newton methods where the optimal
variables are the interpolation points were developed
in~\cite{BeaGug2007}.  Such local optimization methods allow for local
superlinear convergence to local minimizers of the error function, but
cannot guarantee global convergence to the global minimizer. The
analysis of Section~\ref{robust} also shows that using the
  diagonal canonical form for such algorithms may lack the required
robustness properties.

\section{The discrete-time case}
\label{sec:DT}

\pacomm{I have not checked this section in detail. I will do it if
  the paper is not rejected...}

Now consider the equivalent formulation in the discrete-time case.
We then have the dynamical systems
\eqnn
\left\{ \begin{array}{l}x_{k+1} = Ax_k + Bu_k\\y_k= Cx_k\end{array} \right. \quad \mathrm{and} \quad 
\left\{\begin{array}{l}\hx_{k+1} = \hA \hx_k +\hB u\\\hy_k= \hC \hx_k \end{array}\right.
\eeqnn
with transfer functions
$$H(z) = C(zI-A)^{-1}B, \quad \mathrm{and} \quad \hH(z) =\hC(zI-\hA)^{-1}\hB.$$

The squared ${\cal H}_2$-norm of the error function $E(z):= H(z)-\hH(z)$ is then defined as
\e
\calJ := \parallel E(z)\parallel^2_{{\cal H}_2} := \trace\int^\infty _{-\infty}E(e^{j\omega})E(e^{j\omega})^H \frac{d\omega}{2\pi}
= \trace\sum^\infty_{k=0} (C_e A^k_e B_e)(C_e A^k_e B_e)^T
\ee
where $(A_e, B_e, C_e)$ defined in (\ref{AeBeCe}) is again a realization of the error transfer function 
$E(z)$. The ${\cal H}_2$-norm can now be rewritten in terms of the solutions of the Stein equations
\begin{equation}  \label{eq:Stein}
A_e P_e A_e^T + B_eB_e^T =P_e, \quad
A_e^TQ_e A_e + C_e^TC_e =Q_e
\end{equation}
as
\[ \calJ = \trace \left(C_eP_eC_e^T\right) = \trace \left(B_e^TQ_eB_e\right). \]

Partition again the solutions 
\[ P_e := \bma{cc} P & X\\X^T & \hP\ema, \quad Q_e := \bma{cc}Q&Y\\Y^T &\hQ\ema, \]
to obtain the Stein equations in the form
\eqnn
\bma{cc}A&\\&\hA\ema \bma{cc}P&X\\X^T&\hP\ema \bma{cc}A^T&\\&\hA^T\ema+\bma{c}B\\\hB \ema \left[B^T \;\; \hB^T\right] = \bma{cc}P&X\\X^T&\hP\ema ,\\
\bma{cc}A^T&\\&\hA^T\ema \bma{cc} Q&Y\\Y^T & \hQ\ema \bma{cc}A&\\&\hA\ema +\bma{c}C^T\\-\hC^T \ema \left[C \;\;- \hC\right] = \bma{cc}Q&Y\\Y^T&\hQ\ema .
\eeqnn

\begin{thrm} \label{gradientsD}
The gradients $\nabla_\hA \calJ$, $\nabla_\hB \calJ$ and $\nabla_\hC \calJ$ of 
$\calJ:=\| E(s)\|^2_{\calH_2}$ are given by 
\begin{equation} \label{Dgradients} 
\nabla_\hA \calJ = 2(\hQ\hA\hP+Y^TAX), \quad 
\nabla_\hB \calJ = 2(\hQ\hB+Y^TB), \quad  
\nabla_\hC \calJ = 2(\hC\hP-CX), 
\end{equation}
where
\begin{equation} \label{DSyl1}A^TY\hA-C^T\hC=Y, \quad \hA^T\hQ\hA +\hC^T\hC=\hQ, \end{equation}
\begin{equation} \label{DSyl2}\hA X^TA^T +\hB B^T= X^T, \quad \hA\hP\hA^T+\hB\hB^T=\hP. \end{equation}
\end{thrm}

Setting the gradient of $\calJ$ to zero yields the stationarity
  conditions derived in~\cite{BunKubVosWil2007-TR}. These are the
  discrete-time counterpart of Wilson's conditions (see~\cite{Wil1970}
  or Theorem~\ref{gradientsJ}).

Again, at a stationary point (where all gradients are zero) we have that the projection matrices
$$
W:= -Y\hQ^{-1} , \quad V:= X \hP^{-1}
$$
satisfy $\hA = W^TAV$, $\hB=W^TB$, $\hC = CV$, $W^T V = I$ and the Sylvester equations
$$
\left\{\begin{array}{l}
\hA^T(\hQ W^T)A + \hC^T C = (\hQ W^T)\\A(V\hP)\hA^T + B \hB^T = (V\hP)\end{array}\right.
$$
indicating that we are solving a tangential interpolation problem in the inverses of the eigenvalues of $\hA$, 
and this both left and right. 

\medskip

Let us now look at the tangential interpolation conditions for the discrete-time case.
We treat immediately the higher-order case and specialize afterward to the case
of order 1 interpolation conditions. Lemmas \ref{F} and  \ref{FT} have the following 
analogues.

\begin{lmm} \label{FD}
If $\lambda^{-1}$ is not an eigenvalue of $A$, the solution of the matrix equation 
\[A^TYF-Y=C^TL \quad \mathrm{with} \quad 
F:=\begin{bmatrix}\lambda & -1 & \\ & \lambda  & \ddots \\
 & & \ddots & -1 \\ & & & \lambda \end{bmatrix} \in \C^{k\times k},
\]
with $L:=\begin{bmatrix} \ell_0 & \ell_1 & \ldots & \ell_{k-1} \end{bmatrix}$, 
is given by
\[ Y=\begin{bmatrix} (\lambda A^T-I)^{-1}C^T & \ldots & {A^T}^{k-1}(\lambda A^T-I)^{-k}C^T
\end{bmatrix}
\begin{bmatrix} \ell_0 & \ell_1 & \ldots & \ell_{k-1}\\ & \ell_0 & \ddots & \vdots \\ & & \ddots & \ell_1 \\ 
& & & \ell_0\end{bmatrix}.\]
Moreover, let 
\[  \phi_\lambda(z) :=\begin{bmatrix} 1 & (\lambda-z) & \ldots & (\lambda-z)^{k-1} \end{bmatrix}^T, 
\quad y(z):=Y\phi_\lambda(z)\]
then 
\[ y(z)= (zA^T-I)^{-1}C^TL\phi_\lambda(z) + O(\lambda-z)^k\]
which means that the $i$th column $y_i$ of $Y$ is also the coefficient of $(\lambda-z)^{i-1}$ in the Taylor expansion
of $(zA^T-I)^{-1}C^TL\phi_\lambda(z)$. 
\end{lmm}
\begin{proof} The first part easily follows from $(\lambda A^T-I)y_1=C^T\ell_0$ and  $(\lambda A^T-I)y_i=C^T\ell_{i-1}+A^Ty_{i-1}, \; i>1$.
The second part follows from the identity
\[ (zA^T-I)^{-1}C^T = \sum_{i=0}^{\infty}(\lambda-z)^{i}{A^T}^{i}(\lambda A^T-I)^{-i-1}C^T \]
and from the convolution of this formal series with the polynomial vector $L\phi_\lambda(z)$. 
\end{proof}

\medskip

We give the dual version of this lemma without proof. 

\begin{lmm} \label{FDT}
If $\lambda^{-1}$ is not an eigenvalue of $A$, the solution of the matrix equation 
\[FX^HA^T-X^H=R^HB^T 
\]
with $F\in \C^{k\times k}$ as above and $R:=\begin{bmatrix}
    r_{k-1} & r_{k-2} & \ldots & r_0 \end{bmatrix}$, 
is given by
\[ X^H=
\begin{bmatrix} r_0^H & r_1^H & \ldots & r_{k-1}^H\\ & r_0^H & \ddots & \vdots \\ & & \ddots & r_1^H \\ 
& & & r_0^H \end{bmatrix}
\begin{bmatrix} B^T{A^T}^{k-1}(\lambda A^T-I)^{-k}\\ \vdots \\ B^TA^T(\lambda A^T-I)^{-2} \\ B^T(\lambda A^T-I)^{-1}
\end{bmatrix}.\]
Moreover, let 
\[  \psi_\lambda(z) :=\begin{bmatrix} (\lambda-z)^{k-1} & \ldots & (\lambda-z) & 1 \end{bmatrix}, 
\quad x^H(z):=\psi_\lambda(z)X^H\]
then 
\[ x^H(z)= \psi_\lambda(z)R^HB^T(zA^T-I)^{-1} + O(\lambda-z)^k\]
which means that the $i$th row $x_i^H$ of $X^H$ is also the coefficient of $(\lambda-z)^{i-1}$ in the 
Taylor expansion of $\psi_\lambda(z)R^HB^T(zA^T-I)^{-1}$. 
\end{lmm}

\medskip

This now leads to the following theorems with interpolation conditions in terms of the
transfer function $H_*(z):=z^{-1}H^T(z^{-1})$ :
\[  H_*(z):=B^T(I-zA^T)^{-1}C^T = - \sum_{i=0}^{\infty}(\lambda-z)^{i}B^T{A^T}^{i}(\lambda A^T-I)^{-i-1}C^T.\]
Since the proof is essentially the same as the one for the continuous-time case, it is omitted here. 
\begin{thrm} \label{HighDisc}
Let $\hH(z)=\sum_{i=1}^\ell \hH_i(z), \quad \hH_i(z):=\hC_i(zI-\hA_i)^{-1}\hB_i^H$ where 
$\{(\hA_i,\hB_i^H, \hC_i): i=1,\ldots,\ell\}$ is a self-conjugate set and $\hA_i$ is just one Jordan
block of size $k_i$ associated with eigenvalue $\ha_i$, and where $\ha_i^{-1}$ is not a pole of $H(z)$ or
$\hH(z)$. Then with
\[ \hb_i(z)^H := \begin{bmatrix} (\ha_i-z)^{k_i-1}& \ldots & (\ha_i-z) & 1 \end{bmatrix}\hB_i^H, \]
\[ \hc_i(z):= \hC_i\begin{bmatrix} 1 & (\ha_i-z) & \ldots & (\ha_i-z)^{k_i-1} \end{bmatrix}^T, \]
we have
\begin{equation} \label{TI2DB}
[H_*^T(z)-\hH_*^T(z)]\hc_i(z) = O(\ha_i-z)^{k_i}, 
\end{equation}
\begin{equation} \label{TI2DC} 
\hb_i(z)^H[H_*^T(z)-\hH_*^T(z)] = O(\ha_i-z)^{k_i},
\end{equation}
\begin{equation} \label{TI2DA} 
\hb_i(z)^H[H_*^T(z)-\hH_*^T(z)]\hc_i(z) = O(\ha_i-z)^{2k_i},
\end{equation}
where $S_i,T_i$ are as defined in \eqref{ST}.
\end{thrm}

In the case of first-order poles, the conditions reduce to the
  following result, found in~\cite{BunKubVosWil2007-TR} in an
  equivalent form.
\begin{cor}
For the case of first-order poles (i.e. $k_i=1$), the above conditions become~:
\[ [H_*^T(z)-\hH_*^T(z)]\hc_i = O(\ha_i-z), \quad \hb_i^H[H_*^T(z)-\hH_*^T(z)]=O(\ha_i-z),\] 
\[ \hb_i^H[H_*^T(z)-\hH_*^T(z)]\hc_i =O(\ha_i-z)^2.\]
If, moreover, $m=p=1$, we retrieve the $2n$ conditions described in the SISO result of~\cite{MeiLue1967}~:
\[  H_*(\ha_i)=\hH_*(\ha_i), \quad  \frac{d}{dz}\left.H_*(z)\right|_{z=\ha_i}=\frac{d}{dz}\left.\hH_*(z)\right|_{z=\ha_i}, \quad i=1, \ldots, n.\]
\end{cor}

\section{Conclusion}
\label{sec:conc}

In this paper, we have characterized the stationary points of the
$\calH_2$-norm approximation error $\|H(s)-\hH(s)\|_{\calH_2}^2$ in
the MIMO case, with the reduced-order system $\hH(s)$ in Jordan
canonical form. The stationarity conditions take the form of
tangential interpolation conditions---whose degree depend on the size
of the Jordan blocks---written in terms of the Jordan parameters of
$\hH(s)$. The conditions are thus implicit, which calls for iterative
algorithms. However, we have shown that the Jordan-based approach
becomes ill-conditioned in the neighborhood of target transfer
functions $\hH(s)$ with higher-order poles. It is therefore more
robust to use the interpolation conditions in the Sylvester
equation form (Theorem~\ref{projecthat}) since the $\calH_2$ norm is
smooth in the parameters $(\hA,\hB,\hC)$ of these equations. We have
also shown that the underlying optimization problem can have several
local minima by just analyzing the approximation problem by systems of
McMillan degree one (with a real pole) and two (with complex conjugate
poles). The case of discrete-time systems has also been considered.


\bibliography{pabib}

\def\polhk#1{\setbox0=\hbox{#1}{\ooalign{\hidewidth
  \lower1.5ex\hbox{`}\hidewidth\crcr\unhbox0}}}
\begin{thebibliography}{BKVW07}

\bibitem[Ant05]{Ant2005}
Athanasios~C. Antoulas.
\newblock {\em Approximation of large-scale dynamical systems}, volume~6 of
  {\em Advances in Design and Control}.
\newblock Society for Industrial and Applied Mathematics (SIAM), Philadelphia,
  PA, 2005.
\newblock With a foreword by Jan C. Willems.

\bibitem[AW49]{AigWil1949}
P.~R. Aigrain and E.~M. Williams.
\newblock Synthesis of $n$-reactance networks for desired transient response.
\newblock {\em J. Appl. Phys.}, 20:597--600, 1949.

\bibitem[BF79]{ByrFal1979}
Christopher~I. Byrnes and Peter~L. Falb.
\newblock Applications of algebraic geometry in system theory.
\newblock {\em Amer. J. Math.}, 101(2):337--363, 1979.

\bibitem[BG07]{BeaGug2007}
Christopher~A. Beattie and Serkan Gugercin.
\newblock Krylov-based minimization for optimal {$H_2$} model reduction.
\newblock In {\em Proceedings of the 46th IEEE Conference on Decision and
  Control}, 2007.

\bibitem[BGR90]{BalGohRod1990}
Joseph~A. Ball, Israel Gohberg, and Leiba Rodman.
\newblock {\em Interpolation of rational matrix functions}, volume~45 of {\em
  Operator Theory: Advances and Applications}.
\newblock Birkh\"auser Verlag, Basel, 1990.

\bibitem[BKVW07]{BunKubVosWil2007-TR}
A.~{Bunse-Gerstner}, D.~Kubali\'nska, G.~Vossen, and D.~Wilczek.
\newblock $h_2$-norm optimal model reduction for large-scale discrete dynamical
  {MIMO} systems.
\newblock Technical Report 07-04, Universit{\"a}t Bremen, Zentrum f{\"u}r
  Technomathematik, August 2007.
\newblock http://www.math.uni-bremen.de/zetem/reports/reports-liste.html.

\bibitem[Che99]{Che1999LST}
Chi-Tsong Chen.
\newblock {\em Linear System Theory and Design}.
\newblock Oxford University Press, New York, NY, 1999.

\bibitem[GAB07]{GugAntBea2007-u}
S.~Gugercin, A.~C. Antoulas, and C.~Beattie.
\newblock {$H_2$} model reduction for large-scale linear dynamical systems.
\newblock accepted for publication in SIAM J. Matrix Anal. Appl., 2007.

\bibitem[Gan59]{Gan59}
F.~R. Gantmacher.
\newblock {\em The Theory of Matrices {I}, {II}}.
\newblock Chelsea, New-York, 1959.

\bibitem[Gug02]{Gug2002}
Serkan Gugercin.
\newblock {\em Projection methods for model reduction of large-scale dynamical
  systems}.
\newblock PhD thesis, ECE Dept., Rice University, December 2002.

\bibitem[GVV05]{GalVanDoo2004-SIMAX}
K.~Gallivan, A.~Vandendorpe, and P.~{Van~Dooren}.
\newblock Model reduction of {MIMO} systems via tangential interpolation.
\newblock {\em SIAM J. Matrix Anal. Appl.}, 26(2):328--349, 2004/05.

\bibitem[GVV04]{GalVanDoo2004-JCAM}
K.~Gallivan, A.~Vandendorpe, and P.~{Van~Dooren}.
\newblock Sylvester equations and projection-based model reduction.
\newblock {\em J. Comput. Appl. Math.}, 162(1):213--229, 2004.

\bibitem[ML67]{MeiLue1967}
L.~Meier and D.~G. Luenberger.
\newblock Approximation of linear constant systems.
\newblock {\em IEEE Trans. Automatic Control}, 12:585--588, 1967.

\bibitem[VGA08]{DooGalAbs2008}
P.~{Van~Dooren}, K.~A. Gallivan, and P.-A. Absil.
\newblock {$H_2$}-optimal model reduction of {MIMO} systems.
\newblock {\em Appl. Math. Lett.}, 2008.
\newblock to appear.

\bibitem[Wil65]{Wil65}
J.~H. Wilkinson.
\newblock {\em The Algebraic Eigenvalue Problem}.
\newblock Clarendon Press, Oxford, 1965.

\bibitem[Wil70]{Wil1970}
D.~A. Wilson.
\newblock Optimum solution of model-reduction problem.
\newblock {\em Proc. Inst. Elec. Eng.}, 117:1161--1165, 1970.

\end{thebibliography}

\end{document}